\documentclass[11 pt]{amsart}
\usepackage{amssymb,amsmath,epsfig,mathrsfs, enumerate}
\usepackage{graphicx}
\usepackage[normalem]{ulem}
\usepackage{fancyhdr}
\usepackage[margin=3cm]{geometry}
\usepackage{hyperref}
\hypersetup{
 colorlinks   = true,
 urlcolor     = blue,
 linkcolor    = blue,
 citecolor   = red ,
 bookmarksopen=true
}

\theoremstyle{plain}
\newtheorem{thrm}{Theorem}[section]
\newtheorem{lemma}[thrm]{Lemma}
\newtheorem{prop}[thrm]{Proposition}

\newtheorem{rmrk}[thrm]{Remark}
\newtheorem{dfn}[thrm]{Definition}

\newtheorem*{hyp}{HYPOTHESIS}
\begin{document}
% begin top matter
% ***************** macroes needed for this paper ************************

\newcommand{\SL}{\mathcal L^{1,p}( D)}
\newcommand{\Lp}{L^p( Dega)}
\newcommand{\CO}{C^\infty_0( \Omega)}
\newcommand{\Rn}{\mathbb R^n}
\newcommand{\Rm}{\mathbb R^m}
\newcommand{\R}{\mathbb R}
\newcommand{\Om}{\Omega}
\newcommand{\Hn}{\mathbb H^n}
\newcommand{\aB}{\alpha B}
\newcommand{\eps}{\ve}
\newcommand{\BVX}{BV_X(\Omega)}
\newcommand{\p}{\partial}
\newcommand{\IO}{\int_\Omega}
\newcommand{\bG}{\boldsymbol{G}}
\newcommand{\bg}{\mathfrak g}
\newcommand{\bz}{\mathfrak z}
\newcommand{\bv}{\mathfrak v}
\newcommand{\Bux}{\mbox{Box}}
\newcommand{\e}{\ve}
\newcommand{\X}{\mathcal X}
\newcommand{\Y}{\mathcal Y}
\newcommand{\W}{\mathcal W}
\newcommand{\la}{\lambda}
\newcommand{\vf}{\varphi}
\newcommand{\rhh}{|\nabla_H \rho|}
\newcommand{\Ba}{\mathcal{B}_\gamma}
\newcommand{\Za}{Z_\beta}
\newcommand{\ra}{\rho_\beta}
\newcommand{\na}{\nabla_\beta}
\newcommand{\vt}{\vartheta}
\newcommand{\La}{\mathcal{L}}

\numberwithin{equation}{section}

\newcommand{\RN} {\mathbb{R}^N}
\newcommand{\Sob}{S^{1,p}(\Omega)}
\newcommand{\Dxk}{\frac{\partial}{\partial x_k}}
\newcommand{\Co}{C^\infty_0(\Omega)}
\newcommand{\Je}{J_\ve}
\newcommand{\beq}{\begin{equation}}
\newcommand{\bea}[1]{\begin{array}{#1} }
\newcommand{\eeq}{ \end{equation}}
\newcommand{\ea}{ \end{array}}
\newcommand{\eh}{\ve h}
\newcommand{\Dxi}{\frac{\partial}{\partial x_{i}}}
\newcommand{\Dyi}{\frac{\partial}{\partial y_{i}}}
\newcommand{\Dt}{\frac{\partial}{\partial t}}
\newcommand{\aBa}{(\alpha+1)B}
\newcommand{\GF}{\psi^{1+\frac{1}{2\alpha}}}
\newcommand{\GS}{\psi^{\frac12}}
\newcommand{\HFF}{\frac{\psi}{\rho}}
\newcommand{\HSS}{\frac{\psi}{\rho}}
\newcommand{\HFS}{\rho\psi^{\frac12-\frac{1}{2\alpha}}}
\newcommand{\HSF}{\frac{\psi^{\frac32+\frac{1}{2\alpha}}}{\rho}}
\newcommand{\AF}{\rho}
\newcommand{\AR}{\rho{\psi}^{\frac{1}{2}+\frac{1}{2\alpha}}}
\newcommand{\PF}{\alpha\frac{\psi}{|x|}}
\newcommand{\PS}{\alpha\frac{\psi}{\rho}}
\newcommand{\ds}{\displaystyle}
\newcommand{\Zt}{{\mathcal Z}^{t}}
\newcommand{\XPSI}{2\alpha\psi \begin{pmatrix} \frac{x}{|x|^2}\\ 0 \end{pmatrix} - 2\alpha\frac{{\psi}^2}{\rho^2}\begin{pmatrix} x \\ (\alpha +1)|x|^{-\alpha}y \end{pmatrix}}
\newcommand{\Z}{ \begin{pmatrix} x \\ (\alpha + 1)|x|^{-\alpha}y \end{pmatrix} }
\newcommand{\ZZ}{ \begin{pmatrix} xx^{t} & (\alpha + 1)|x|^{-\alpha}x y^{t}\\
     (\alpha + 1)|x|^{-\alpha}x^{t} y &   (\alpha + 1)^2  |x|^{-2\alpha}yy^{t}\end{pmatrix}}
\newcommand{\norm}[1]{\lVert#1 \rVert}
\newcommand{\ve}{\varepsilon}
\title{}
\title[carleman estimates etc.]{  Carleman estimates for a class of variable coefficient degenerate elliptic operators  with applications to   unique continuation}
\author{Agnid Banerjee}
\address{Tata Institute of Fundamental Research\\
Centre For Applicable Mathematics \\ Bangalore-560065, India}\email[Agnid Banerjee]{agnidban@gmail.com}

\author{Ramesh Manna}
\address{Department of Mathematics, Indian Institute of Science, 560 012 Bangalore, India}
\email{rameshmanna@iisc.ac.in}

\thanks{A.B is supported in part by SERB Matrix grant MTR/2018/000267 and by Department of Atomic Energy,  Government of India, under
project no.  12-R \& D-TFR-5.01-0520.}
\thanks{R.M is  supported by  C.V. Raman PDF, R(IA)CVR-PDF/2020/224}
%\thanks{Third author is  supported by  SERB National Postdoctoral fellowship, PDF/2017/0027}

%\dedicatory{Dedicated to Giovanni Alessandrini, on his $60$-th birthday, with great affection and admiration}

%
% 
% AMS information
%
\keywords{}
\subjclass{}

\maketitle

\tableofcontents
% end top matter
%\begin{abstract}
%Based on a variant of the frequency function approach of  Almgren, we  establish  an optimal  bound on the  vanishing order of solutions to stationary Schr\"odinger equations associated to a class of subelliptic equations with variable coefficients whose model is the so-called Baouendi-Grushin operator. Such bound provides a quantitative form of strong unique continuation that can be thought of as an analogue of the recent results of Bakri and Zhu for the standard Laplacian.
%\end{abstract}

\begin{abstract}
In this paper, we obtain   new Carleman estimates for a class of variable coefficient degenerate elliptic operators whose constant coefficient model at one point is the so called Baouendi-Grushin operator.  This generalizes the  results  obtained by the two of  us with Garofalo in \cite{BGM} where similar estimates were established for the "constant coefficient"   Baouendi-Grushin operator. Consequently,  we obtain: (i)  a Bourgain-Kenig  type quantitative uniqueness result in the variable coefficient setting; (ii) and a strong unique continuation property for a class of degenerate sublinear equations. We also derive a subelliptic version   of a  scaling critical Carleman estimate proven by   Regbaoui  in the Euclidean setting   using which we deduce a new unique continuation result in  the case of scaling critical Hardy type potentials.

\end{abstract}

\section{Introduction}\label{S:intro}
In this paper, we study some ad-hoc $L^{2}-L^{2}$ Carleman estimates for operators of the type
\begin{equation}\label{operator}
\La= \sum_{i=1}^N X_i(a_{ij} (z, t) X_j ),
\end{equation}
where $(z,t) \in  \R^m \times \R^k$, $N=m+k$ and the vector fields $X_1, ..., X_N$ are given by
\begin{equation}\label{vector}
X_i= \partial_{z_i}, i=1,..,m \quad X_{m+j}= |z|^{\gamma}\partial_{t_j}, j=1,..,k, \\ \gamma>0.\end{equation}
Besides ellipticity, the $N \times N$ matrix valued function $A(z,t)= [a_{ij}(z,t)]$ is required to satisfy some structural assumptions that will be specified in \eqref{H} in Section \ref{pre} below.  Such conditions reduce to the standard Lipschitz continuity when the dimension $k=0$ or when $\gamma=0$.  One should note that when $A= \mathbb{I}$, the operator in \eqref{operator} reduces to the well known Baouendi-Grushin operator given by 

\begin{equation}\label{ba}
\Ba = \Delta_z  + |z|^{2\gamma} \Delta_t.
\end{equation}
The operator $\Ba$ is degenerate  elliptic along $\{z=0\}$ and it is not translation invariant in $\R^N$. We recall that a more general class of operators modelled on $\Ba$ was first introduced by Baouendi who studied the Dirichlet problem in some appropriate weighted Sobolev space in \cite{Ba}.  Subsequently  in \cite{Gr1}, \cite{Gr2}, Grushin analyzed the hypoellipticity of this operator when $\gamma \in \mathbb{N}$. We also refer to \cite{FL0, FL, FL3, W} for other interesting works related to $\Ba$.  Remarkably, the operator $\Ba$ also plays an important role in the work \cite{KPS} on the higher regularity of the free boundary in the classical Signorini problem.   

To provide the reader with some perspective we mention that when $\gamma = 1$ the operator $\Ba$ is intimately connected to the sub-Laplacians in groups of Heisenberg type. In such Lie groups, in fact, in the exponential coordinates with respect to a fixed orthonormal basis of the Lie algebra, the sub-Laplacian  is given by
\begin{equation}\label{slH}
\Delta_H = \Delta_z + \frac{|z|^2}{4} \Delta_t  + \sum_{\ell = 1}^k \p_{t_\ell} \sum_{i<j} b^\ell_{ij} (z_i \p_{z_j} -  z_j \p_{z_i}),
\end{equation}
where $b^\ell_{ij}$ indicate the group constants. If $u$ is a solution of $\Delta_H$ that further annihilates the symplectic vector field $\sum_{\ell = 1}^k \p_{t_\ell} \sum_{i<j} b^\ell_{ij} (z_i \p_{z_j} -  z_j \p_{z_i})$, then, up to a normalisation factor of $4$, $u$ solves  the operator $\Ba$ obtained by letting $\gamma = 1$ in \eqref{ba} above.

Concerning the question of interest for this paper, the unique continuation property, we mention that for general uniformly elliptic equations there are essentially two known methods  for proving it. The former  is  based on Carleman inequalities, which are appropriate weighted versions of  Sobolev-Poincar\'e inequalities. This method  was first introduced by T. Carleman in his fundamental work \cite{C} in which  he  showed that  strong unique continuation holds  for equations of the type 
\[
-\Delta u +V u = 0, \  \ \ \ \ \  V \in L^{\infty}_{loc}(\R^2).
\]
In his pioneering work \cite{A}, Aronszajn extended such estimates to higher dimensions and uniformly elliptic operators with $C^{2}$ principal part. Subsequently, in \cite{AKS} the authors generalised this result to uniformly elliptic equations with Lipschitz coefficients in the principal part, see also \cite{Ho}. We stress that unique continuation fails in general when the coefficients of the principal part are only H\"older continuous, see Plis' counterexample in \cite{Pl}, and also \cite{Mi}. 
The second approach came up in the works of Garofalo and  Lin, see \cite{GL1}, \cite{GL2}. Their method is based on the almost monotonicity of a generalisation of the frequency function, first introduced by Almgren in \cite{Al} for harmonic functions. Using this approach, they were able to obtain new quantitative  information on the zero set of solutions to divergence form elliptic equations with Lipschitz coefficients.

The unique continuation property for the degenerate operators $\Ba$ is much subtler than the one for the Laplacian. It was first established by Garofalo in  \cite{G}. In that work he introduced a Almgren type  frequency function associated with $\Ba$, and proved that such function is monotone non-decreasing on solutions of $\Ba= 0$. These results were extended to more general variable coefficient equations by Garofalo and  Vassilev  in \cite{GV}. One should also see the related works \cite{GLan} and \cite{GR} on the Heisenberg and more general Carnot groups. We also note that a version of the monotonicity formula for $\Ba$ played an extensive role in the recent work \cite{CSS} on the obstacle problem for the fractional Laplacian. 

Using some ad hoc Carleman estimates in \cite{GarShen} the authors were able to establish for the first time some strong unique continuation results for $\Ba u+Vu=0$ in the difficult situation when $V$ satisfies appropriate $L^p$ integrability hypothesis. Their analysis, which is closer in spirit to the works \cite{JK}, \cite{J}, \cite{ChS}, \cite{KT} to name a few,  only covers the special case when $\gamma=k=1$ in \eqref{ba}, and ultimately rests on delicate boundedness properties of certain projector operators generalising some of the results in \cite{So}.  We also refer to the recent work of one of us with Mallick in  \cite{BM} where, using such projector operator estimates, a new $L^{2}-L^{2}$ Carleman estimate is derived. Using the latter, the authors deduce strong unique continuation when the potential $V$  satisfies Hardy type growth assumptions. It is worth mentioning at this point that the general situation of the results in \cite{GarShen} presently remains a challenging open question.

$L^{2}-L^2$ Carleman estimates  with singular weights  for the general  Baouendi-Grushin operators  $\Ba$ which are analogous to the ones in  \cite{A} have been established  very recently by two of us with Garofalo in \cite{BGM}  by using elementary arguments based on integration by parts and by an appropriate  application of Rellich type identity.  In the same paper,    quantitative uniqueness result of Bourgain-Kenig type (see  \cite{BK}) and  a strong unique   continuation for a class of  sublinear equations  of the type \eqref{sub} (when $A\equiv \mathbb{I}$) are also  proven. 

In the present work, we generalize the results in \cite{BGM} to  variable coefficient principal part where the matrix  valued function $A$  is assumed to be Lipschitz continuous with respect to a suitable pseudo-distance associated to the system of vector fields $\{X_i\}$.  We refer to \eqref{H} below for the precise assumptions. This framework was first introduced by Garofalo and Vassilev in the above cited paper \cite{GV}.  It is easily seen that  in the situation when $k=0$  the hypothesis \eqref{H} below  coincides with the usual Lipschitz continuity at the origin of the coefficients $a_{ij}$.  Our Carleman estimates thus encompass those in the cited paper \cite{AKS}. Our main results Theorem \ref{main} and Theorem \ref{main2}  can be seen as the variable coefficient analogues of the corresponding results in \cite{BGM}.  The key to the proof of such results are the Carleman estimates in \eqref{est1} and \eqref{f10} below that we derive. As the reader will see, the proof of these estimates are more involved than that for $\Ba$ because of the additional error terms that are incurred due to the Lipschitz perturbation of the principal part. Such error terms are eventually handled by a delicate interpolation type argument  in the proof of the respective estimates.  As an application of our techniques, we also show how to obtain a further refined estimate for  zero-order $C^{1}$  perturbations of the operator  as in \eqref{df} below which in particular implies a quantitative uniqueness result of Donnelly-Fefferman type (see Theorem \ref{DF1}). We mention that the result in Theorem \ref{DF1} has however   been previously obtained by one of us with Garofalo in \cite{BG1} by an adaptation of the Almgren's frequency function approach. Therefore this part of our work can be thought of as an alternate approach to  the  Donnelly-Fefferman type quantitative uniqueness in this degenerate setting.  As a further extension of our techniques, we also  establish a subelliptic version of a  critical Carleman estimate proven by Regbaoui in \cite{Reg}  for uniformly elliptic operators which in turn implies a certain unique continuation result for equations of the type \eqref{hj} where the potential $V$ satisfies the Hardy type growth assumption  as in \eqref{V1}  (see Theorem \ref{har1}).  We mention that proof of the corresponding estimate in \cite{Reg} uses in a crucial way the polar   decomposition of the  frozen constant coefficient operator. Our  proof of \eqref{har1}  is quite  different from that in \cite{Reg} and is instead based on a suitable  adaptation of a  Rellich type  identity  as stated in \eqref{re} below.  Therefore in that sense, the proof of all the Carleman estimates in this paper have a universal character. 
  Over here, we would like to    mention that   Theorem \ref{har1}  is however  slightly  weaker than the   strong unique continuation property because the hypothesis of the theorem  involves a somewhat different notion of vanishing ( see \eqref{vp1} below). Nevertheless  it   provides an improvement of Theorem 4.4 in \cite{G}. We refer to Section \ref{pre} for further discussions on this topic.  Finally we would like to point out that a somewhat technical level, our work  also differs  additionally from \cite{GV} and \cite{BG1} (which uses the frequency function approach in this variable coefficient setting)  in the sense that   for the proof of the Carleman estimates, a third derivative estimate of the gauge function $\rho$  as  in Lemma \ref{lma3.3} below is crucial for our analysis.  We provide a proof of such an estimate in the Appendix because it involves a long and delicate computation.
  
  The paper is organized as follows. In Section \ref{pre}, we introduce  relevant notions, gather some known results and  then  state our main results. In Section \ref{mn}, we   prove our main results. In the Appendix, we give a proof of  Lemma \ref{lma3.3}.
  
  \textbf{Acknowledgment:} The authors would like to thank Nicola Garofalo for sharing ideas and  discussions at various stages of the work. 
  
  \section{Notations and preliminary results}\label{pre}
Henceforth in this paper we follow the notations adopted in \cite{G} and \cite{GV}, with one notable proviso: the parameter $\gamma>0$ in \eqref{ba}, etc. in this paper plays the role of $\alpha >0$ in \cite{G} and \cite{GV}. The reason for this is that we have reserved the greek letter $\alpha$ for the powers of the singular  weights in our Carleman estimates.  Throughout the paper, whenever convenient, we will use the summation convention over repeated indices. 
Given a function $f$, we respectively denote
\begin{equation*}
Xf= (X_1f,...,X_Nf),\ \ \ \ \ \ \ \ \ |Xf|^2= <Xf,Xf> = \sum_{i=1}^N (X_i f)^2,
\end{equation*}

the intrinsic (degenerate) gradient of a function $f$, and the square of its length where the vector fields $\{X_i\}$ are defined as in \eqref{vector}. We note that the vector fields $X_i$ are homogeneous of degree one with respect to the following family of anisotropic dilations 
\begin{equation}\label{dil}
\delta_\la(z,t)=(\la z,\la^{\gamma+1} t),\ \ \ \ \ \ \ \ \la>0.
\end{equation}
Consequently,  the Baouendi-Grushin operator $\Ba$  as defined  in \eqref{ba} is homogeneous of degree two with respect to \eqref{dil}. Let $dzdt$ denote the Lebesgue measure in $\R^N$. Since $d(\delta_\la(z,t)) = \la^Q dz dt$, where
\begin{equation}
Q= m + (\gamma+1) k,
\end{equation}
such number plays the role of a dimension in the analysis of the operator
$\Ba$ as well as $\La$ as in \eqref{operator}. 
For instance, one has the following remarkable fact (see \cite{G}) that the fundamental solution $\Gamma$ of $\Ba$  with pole at the origin is given by the formula
\[
\Gamma(z,t) = \frac{C}{\rho(z,t)^{Q-2}},\ \ \ \ \ \ \ \ \ (z,t)\not= (0,0),
\]
where $C>0$ is suitably chosen and $\rho$ is the pseudo-gauge 
\begin{equation}\label{rho}
\rho(z,t)=(|z|^{2(\gamma+1)} + (\gamma+1)^2 |t|^2)^{\frac{1}{2(\gamma+1)}}.
\end{equation}
A  function $v$ is $\delta_{\la}$-homogeneous of degree $\kappa$ if and only if $Zv=\kappa v$. Since $\rho$ in \eqref{rho} is homogeneous of degree one, we have
\begin{equation}\label{hg}
Z\rho=\rho.
\end{equation}
We respectively denote by 
\[
B_r = \{(z,t)\in \R^N\mid \rho(z,t) < r\},\ \ \ \ \ \ \ \ S_r = \{(z,t)\in \R^N\mid \rho(z,t) = r\}, 
\]
the gauge pseudo-ball and sphere centered at $0$ with radius $r$. The infinitesimal generator of the family of dilations \eqref{dil}  is given by the vector field
\begin{equation}\label{Z}
Z= \sum_{i=1}^m z_i \partial_{z_i} + (\gamma+1)\sum_{j=1}^k t_j \partial_{t_j}.
\end{equation}
We note the important facts that
\begin{equation}\label{divZ}
\operatorname{div} Z = Q,\ \ \ \ \ \ \ \ \ \ \ \ [X_i,Z] = X_i,\ \ \ i=1,...,N.
\end{equation}
We also need the angle function $\psi$ introduced in \cite{G}
\begin{equation}\label{psi}
\psi = |X\rho|^2= \frac{|z|^{2\gamma}}{\rho^{2\gamma}}.
\end{equation}

The function $\psi$ vanishes on the characteristic manifold $M=\Rn \times \{0\}$ and clearly satisfies $0\leq \psi \leq 1$. Since $\psi$ is homogeneous of degree zero with respect to \eqref{dil}, one has
\begin{equation}\label{Zpsi}
 Z\psi = 0.
 \end{equation}
 If $f\in C^2(\R)$ and $v\in C^2(\R^N)$, then we have the important identities (see \cite{G})
 \begin{equation}\label{ii}
 \Ba f(\rho) = \psi \left(f''(\rho) + \frac{Q-1}{\rho} f'(\rho)\right),
 \end{equation}
 and 
 \begin{equation}\label{h10}
<Xv,X\rho> = \sum_{i=1}^N X_i v X_i \rho = \frac{\psi}{\rho}  Zv.
\end{equation}
Henceforth, for any two vector fields $U$ and $W$, $[U,W] = UW - WU$ denotes their commutator. 

A first basic assumption on the matrix-valued function $A=[a_{ij}]$ is that it be symmetric and  uniformly elliptic. i.e., $a_{ij} = a_{ji}$, $i,j=1,...,N$, and there exists $\lambda > 0$ such that for every $(z,t)\in \R^N$ and $\eta\in \R^N$ one has
\begin{equation}\label{ue}
\lambda|\eta|^2 \leq<A(z,t)\eta, \eta> \leq \lambda^{-1} |\eta|^2.
\end{equation}

Throughout the paper we assume that 
\begin{equation}\label{A0}
A(0,0) = I_N,
\end{equation}
where $I_N$ indicates the identity matrix in $\R^N$.
In order to state our main assumptions (H) on the matrix $A$ it
will be useful to represent  the latter in the following block
form
\begin{equation*}
A= \begin{pmatrix} A_{11} & A_{12} \\ A_{21} & A_{22}
\end{pmatrix},
\end{equation*}
Here, the entries are respectively $ m\times
m,\thickspace m\times k,\thickspace k\times m$ and $k\times k $
matrices, and we assume that $A^t_{12}=A_{21}$.  We shall denote
by $B$ the matrix
\[
B = A - I_N,
\]
and thus
\begin{equation}\label{Bat0}
B(0,0) = O_N,
\end{equation}
thanks to \eqref{A0}.  We now state the structural assumptions on the matrix $A$.
\begin{hyp}
	There exists a positive constant $\Lambda$ such that one has in $B_{1}$ the
	following estimates
	\begin{equation*}
	|b_{ij}| = |a_{ij} - \delta_{ij}|\ \leq\
	\begin{cases}
	\Lambda\rho, \hskip1.4truein \text{ for }\ 1\leq i,\, j\leq m,\\
	\\
	\Lambda \psi^{\frac12+\frac1{2\gamma}}\rho\, =\, \Lambda\frac{|z|^{\gamma+1}}{\rho^\gamma},\quad
	\text{otherwise},
	\end{cases}
	\end{equation*}
	\begin{equation}\label{H}
	\tag{H}
	\end{equation}
	\begin{equation*}
	|X_kb_{ij}| = |X_ka_{ij}|\ \leq\
	\begin{cases}
	\Lambda, \hskip1.1truein \text{ for }\quad  1\leq k\leq m ,\ \text{ and }\ 1\leq i,\,j\leq m
	\\
	\Lambda \psi^{1+\frac{1}{2\gamma}}\quad \text{when $k >m$ and $\max\{i, j\} > m$}
	\\
	\Lambda \psi^{1/2}\quad
	\text{otherwise}.
	\end{cases}
	\end{equation*}
\end{hyp}	

\

An interesting typical example of a matrix $A$ satisfying the conditions \eqref{H} is

\begin{equation*}
A= \begin{pmatrix} 1 + \rho f(z,t) & |z|^{\gamma+1} g(z,t) \\ |z|^{\gamma+1} g(z,t) & 1+ |z|^{\gamma+1} h(z,t)
\end{pmatrix},
\end{equation*}
where $f,g$ and $h$ are Lipschitz continuous near the origin in $\R^2$. In this example $m=k=1$.

We next collect several preliminary results that will be important in our proof. We first consider the quantity
\begin{eqnarray} \label{defmu}
\mu=\langle A X \rho, X\rho\rangle.
\end{eqnarray}
 In view of the uniform ellipticity of $A$, we have
 
 \begin{eqnarray} \label{musi}
 \lambda \psi \leq \mu \leq \lambda^{-1} \, \psi.
 \end{eqnarray}
The following vector field $F$ will play an important role in the paper:
\begin{eqnarray} \label{defF}
F=\frac{\rho}{\mu} \sum_{i,j=1}^Na_{ij} X_i \rho X_j.
\end{eqnarray}
We note that 
\begin{eqnarray} \label{3.5}
F\rho= \rho.
\end{eqnarray} 

 \begin{dfn}
 	We have $$B=A-Id, ~~\sigma = \langle B X\rho, X\rho\rangle.$$
 	One more notation we will use is: $(b_{ij})=B.$	
 \end{dfn}

%Let $\sigma=\langle B X \rho, X \rho \rangle =\mu-\psi.$
In the next theorem we collect several important estimates that have been established in \cite{G} and \cite{GV} which will be useful throughout the work. 

\begin{thrm}\label{Est1}
	There exists a constant $C(\beta,\lambda,\Lambda,N)>0$ such that for any function $u$ one has:
	\begin{itemize}
		\item[(i)] $|Q- \operatorname{div} F| \leq C\rho$;
		\item[(ii)] $|F\mu|, |F\psi| \leq C \rho \psi$;
		\item[(iii)] $\operatorname{div} (\frac{\sigma Z}{\mu}) \leq C \rho$;
		\item[(iv)] $|X_{i}\rho|\leq \psi^{1+\frac{1}{2\gamma}}, \ \ i=1,...,m,\ \ \ \ |X_{m+j} \rho| \leq (\gamma+1)\psi^{1/2},\ \  j=1,...,k$;
		\item[(v)] $ |F-Z| \leq C \rho^2$;
		\item[(vi)] $|<FAXu, Xu>| \leq C \rho |Xu|^2$;
		\item[(vii)] $|[X_i,F]u -X_iu| \leq C \rho |Xu|$,\ \ \ \ $i=1,...,N$;
		\item[(viii)] $|\sigma| \leq C \rho \psi^{3/2+ \frac{1}{2\gamma}}\ |X\sigma| \leq C\psi^{3/2}$;
		\item[(ix)] $|\frac{b_{ij} X_j\rho X_i}{\mu}| \leq C|z|$;
		\item[(x)] $|X_i\psi|\leq \frac{C\gamma\psi}{|z|}, i=1,...,m,\ \ \ \ |X_{n+j} \psi| \leq \frac{C\gamma\psi}{\rho}, j=1, ..., k$;
		\item[(xi)] $|\frac{\sigma}{\mu}| \leq C \rho \psi,\  |Z\sigma| \leq C \rho \psi,\ |X_k\sigma| \leq C \psi^{3/2}$;
		\item[(xii)] $|[X_i, -\frac{\sigma Z}{\mu}]u| \leq C \rho |Xu|$,\ \ \ \  \ (Lemma 2.7 in \cite{GV});
		\item[(xiii)] $|[X_\ell, \frac{\rho}{\mu} \sum_{i,j=1}^N \frac{b_{ij} X_j\rho} X_i]u| \leq C \rho |Xu|$,\ \ $\ell = 1,...,N$.
	\end{itemize}
\end{thrm}

We also need the following lemmas.
 \begin{prop}[Proposition 3.1, \cite{GV}] \label{gauge}
 We have
 \begin{eqnarray*}
		X_l\rho=\begin{cases}
			\psi \frac{z_l}{\rho}, & \text{ for } 1\leq l\leq m\\
			(\gamma+1) \, \psi^{1/2} \frac{t_{l-m}}{\rho^{\gamma+1}}, & \text{ for } m+1\leq l\leq N.
		\end{cases}
	\end{eqnarray*} 
 
 Consequently it follows that, 
 \begin{enumerate}
 \item $|X_i\rho| \leq \psi^{1+\frac{1}{2\gamma}}$ for $1\leq i\leq m.$
 
  \item $|X_{n+i}\rho| \leq (\gamma+1)\psi^{\frac{1}{2}}$ for $1\leq i\leq k.$
 
 \end{enumerate}
 \end{prop}
 
 \begin{lemma} \label{secondder}
	We have  the expressions for the second derivatives of $\rho$ (See, Proposition 3.3 in \cite{GV}):	
	\begin{enumerate}
		\item  For $1\leq i,j \leq m,$ we have:
		\begin{eqnarray*}&&X_iX_j \rho=-(2\gamma+1) z_iz_j \frac{\psi^2}{\rho^3}+ \left[2\gamma \frac{z_iz_j}{|z|^2}+\delta_{ij}\right] \frac{\psi}{\rho}. 
		\end{eqnarray*}	
	
	\item For $1\leq i\leq m$ and $1\leq j \leq k,$ we have:
	\begin{eqnarray*}
		&&X_iX_{m+j} \rho=-(2\gamma+1)(\gamma+1) \frac{z_it_j}{|z|^{\gamma}} \frac{\psi^2}{\rho^3} +\frac{\psi}{\rho}\left[\gamma(\gamma+1) \frac{z_it_j}{|z|^{\gamma+2}}\right].
	\end{eqnarray*}

\item For $1\leq i\leq m$ and $1\leq j\leq k$ we have:
\begin{eqnarray*}
X_{m+j} X_i\rho=-(2\gamma+1)(\gamma+1) \frac{z_it_j}{|z|^{\gamma}} \, \frac{\psi^2}{\rho^3}.
\end{eqnarray*}

\item For $1\leq i, j\leq k$ we have:
\begin{eqnarray*}
X_{m+i}X_{m+j} \rho=-(2\gamma+1)(\gamma+1)^2 \frac{t_jt_i}{|z|^{2\gamma}} \, \frac{\psi^2}{\rho^3} +(\gamma+1) \delta_{ij} \frac{\psi}{\rho}. 
\end{eqnarray*}
	\end{enumerate}
	
\end{lemma} 
 
 Lemma \ref{secondder} in particular implies the following bounds.

 \begin{prop}[Proposition 3.3, \cite{GV}] \label{prop3.4}
 	\begin{eqnarray*}
 		&&|X_iX_j\rho| \leq C\frac{\mu}{\rho} \text{ for } 1\leq i, j\leq m \text{ or } m+1\leq i, j \leq N,\\
 		&&|X_iX_{m+j}\rho| \leq C\frac{\mu^{\frac{1}{2}}}{|z|}=C \frac{\mu^{\frac{1}{2}-\frac{1}{2\gamma}}}{\rho} \text{ for } 1\leq i\leq m, ~1\leq  j \leq k,\\
 		&&|X_{m+j}X_i\rho| \leq C\frac{\mu^{\frac{3}{2}|z|}}{\rho^2}=C\frac{ \mu^{\frac{3}{2}+\frac{1}{2\gamma}}}{\rho} \text{ for } 1\leq i\leq m, ~1\leq  j \leq k.
 	\end{eqnarray*}
 \end{prop}
% \begin{lemma}[Lemma 3.5, \cite{GV}]
% 	If \eqref{H} holds then:
% 	\begin{eqnarray*}
% 		&&|\sigma| \leq C \rho \, \mu^{\frac{3}{2}+\frac{1}{2\gamma}},\\ 
% 		&& |X_l\sigma| \leq C  \, \mu^{\frac{3}{2}},~~1\leq l\leq N.
% 	\end{eqnarray*}
% \end{lemma}
 \begin{rmrk} \label{rmk3.7}
 	It is easily seen that $$|X_lb_{ij} X_i \rho| \leq C \psi.$$ 
	We also have (See \cite{GV}, page 653)
	\begin{align}
 \sum_{i,j=1}^N|X_i b_{ij} \, X_j \rho| \leq  C\mu, \notag~~\text{ and }\sum_{i,j=1}^N| b_{ij} \, X_i \, X_j \rho| \leq  C\mu.\notag
  \end{align}
 \end{rmrk}
% \begin{lemma}[Lemma 3.7, \cite{GV}]
% 	If \eqref{H} holds then:
% 	\begin{eqnarray*}
% 		|X_l\left( \frac{\psi}{\mu}\right)\sigma| \leq C  \, \mu^{\frac{1}{2}},~~1\leq l\leq N.
% 	\end{eqnarray*}
% \end{lemma}
% 
% \begin{lemma}[Lemma 3.8, \cite{GV}]
% 	If \eqref{H} holds then:
% 	\begin{eqnarray*}
% 		|X_l\left( \frac{\rho}{\mu}\right)\sigma| \leq C  \, \mu^{-1-\frac{1}{2\gamma}},~~1\leq l\leq N.
% 	\end{eqnarray*}
% \end{lemma}
 
 \begin{lemma}[Lemma 3.9, \cite{GV}] \label{lma3.10}
 	If \eqref{H} holds then:
 	\begin{eqnarray*}
 		|b_{kj}X_j \rho| \leq C  \, \rho \,  \mu^{1+\frac{1}{2\gamma}}.
 	\end{eqnarray*}
 \end{lemma}

% Next, we consider 
%\begin{equation*}
%\rho(z,t)=(|z|^{2(\gamma+1)} + (\gamma+1)^2 |t|^2)^{\frac{1}{2(\gamma+1)}}.
%\end{equation*}	
%Differentiating, we obtain
%\begin{eqnarray*}
%	X\rho=\rho^{-(2\gamma+1)}\begin{pmatrix}
%		|z|^{2\gamma}z\\
%		(\gamma+1)t
%	\end{pmatrix}.
%\end{eqnarray*}	

We also need the following third  derivative estimate in our analysis.  
\begin{lemma} \label{lma3.3}
	Let $F=\frac{\rho}{\mu}\sum a_{qr} X_q\rho X_r.$ Then, we have
	\begin{align}\label{third}
		&|F(b_{ij}X_i X_j \rho)|\leq C \psi.
%		 \, \rho^2 \mu^{\frac{1}{2\gamma}}\left( \frac{\psi}{\rho |z|}+\frac{\psi}{\rho^2}+\frac{\psi^{1+\frac{1}{2\gamma}}}{|z|^2}+\frac{\psi}{\rho |z|}+\frac{\psi^{1+\frac{1}{2\gamma}}}{\rho^2|z|^2}+\frac{\psi^2}{\rho^3|z|}+\frac{\psi^{1+\frac{1}{2\gamma}}}{\rho^2|z|^2}+\frac{\psi}{\rho^3|z|}\right)\\
%		&=C \, \rho^2 \mu^{\frac{1}{2\gamma}}\left( \frac{\psi^{1-\frac{1}{2\gamma}}}{\rho^2}+\frac{\psi}{\rho^2}+\frac{\psi^{1-\frac{1}{2\gamma}}}{\rho^2}+\frac{\psi^{1-\frac{1}{2\gamma}}}{\rho^2 }+\frac{\psi^{1-\frac{1}{2\gamma}}}{\rho^4}+\frac{\psi^{2-\frac{1}{2\gamma}}}{\rho^4}+\frac{\psi^{1-\frac{1}{2\gamma}}}{\rho^4}+\frac{\psi^{1-\frac{1}{2\gamma}}}{\rho^4}\right)\\
%		& \leq C \frac{\psi}{\rho^2}.\notag
	\end{align}
\end{lemma}	
The proof of the Lemma \ref{lma3.3}  which is based on a long computation   is postponed to the appendix.

\medskip

We  then define the relevant function space that is repeatedly  used in our work. 
\begin{dfn}
For a given domain $\Om$,  we denote by  $S^{2,2}(\Om)$, the completion of $C^{\infty}(\overline{\Om})$ under the norm
\[
||f||_{S^{2,2}(\Om)}= \int_{\Om} f^2 + |Xf|^2 + \sum_{i,j=1}^N |X_i X_j f|^2   
\]
We instead indicate with $S^{2,2}_0(\Om)$ the completion of $C^{\infty}_{0}(\Om)$ under the same norm.
\end{dfn}

We now introduce the relevant notion of vanishing to infinite order.

\begin{dfn}\label{v0}
  We  say  that $u$ vanishes to infinite order at the origin $(0,0)$, if for every $\ell>0$ one has
 \begin{align}\label{vanLp}
 \int_{B_r} |u|^2 \psi  = O(r^\ell), \ \ \ \ \ \text{as}\ r \rightarrow 0.
 \end{align}
 \end{dfn}
 
 \begin{rmrk}
Throughout this paper, when we say that a constant is universal, we mean that it depends exclusively on $m, k, \beta$, on the ellipticity bound $\lambda$ on $A(z,t)$, see \eqref{ue} above, and on the Lipschitz bound $\Lambda$ in \eqref{H}. Likewise, we will say that $O(1)$, $O(r)$, etc. are universal if $|O(1)| \le C$, $|O(r)|\le C r$, etc., with $C\ge 0$ universal.
 
 \end{rmrk}

\subsection{Statement of the main results}
We now state the main results of the paper.
Our first result is the  subelliptic analogue of the corresponding quantitative uniqueness result of Bourgain and Kenig for the Euclidean Laplacian, see \cite{BK}. We also refer to the work of Bakri \cite{Bk1} for a generalisation of their result to Laplace Beltrami operators on compact manifolds. From now onwards, by $\La$, we refer to the operator  defined in \eqref{operator} where $A$ satisfies the assumptions in \eqref{H}.

\begin{thrm}\label{main}
Let $u \in S^{2,2}(B_1)$ with $|u| \leq 1$  be a solution to
\begin{equation}
\La u=Vu,
\end{equation}

 where the potential  $V$ satisfies the following bound
\begin{equation}\label{vasump}
|V(z,t)| \leq K \psi.
\end{equation}
 Then, there exists  universal $R_0 \in (0, 1/2]$  and  constants $C_1,C_2$ depending also on  $\int_{B_{\frac{R_0}{4}}}  u^2 \psi$, such that for all $0<r<  \frac{R_0}{9}$ one has
\begin{equation}\label{main1}
||u||_{L^{\infty}(B_r)} \geq C_1 \left(\frac{r}{R_0}\right)^{C_2 (K^{2/3}+1)}.   
\end{equation}
\end{thrm}

It is worth emphasizing that, when $k=0$, we have $N = m$ and then from \eqref{psi} we have $\psi \equiv 1$. In such a case the constant $K$ in \eqref{vasump} can be taken to be  $||V||_{L^{\infty}}$, and Theorem \ref{main} reduces to the cited Euclidean result in \cite{BK}. We note that the sharpness of the estimate corresponding to \eqref{main1} follows from counterexamples due to  Meshkov, see \cite{Me}.

For zero order $C^1$ perturbation of the operator $\La$, we also obtain the following  analogue of a Carleman estimate proven by  Bakri in \cite{Bk} for Laplace Beltrami operators on manifolds. 

\begin{thrm} \label{DF}
	Let $0< \ve < 1$ be fixed. There exists a universal $R_0>0$, depending also  on $\ve$, such that  for all $R \leq R_0$, $u \in S^{2,2}_{0}(B_R)$,  and $V$ satisfying
	\begin{equation}\label{vassump1}
	|V| + |FV| \leq K \psi,
	\end{equation}
	 one has \begin{align}\label{df}
		&\alpha^3\int \rho^{-2\alpha- 4+\ve} u^2 e^{2\alpha \rho^{\ve}} \mu + \alpha  \int_{B_R}  \rho^{-2\alpha-2+\epsilon} e^{2\alpha  \rho^{\epsilon}} \langle AXu, Xu \rangle  dz dt \\
		& \leq  C \int \rho^{-2\alpha} e^{2\alpha \rho^{\ve}} (\La u +V u)^2 \mu^{-1}\notag,
	\end{align}
	for  universal constants $C, C_1>0$  depending also on $\ve$ such that  $$\alpha \geq  C_1 (K^{1/2} +1).$$
\end{thrm}
As a consequence of the estimate \eqref{df}, we deduce the following quantitative uniqueness result for "$C^{1}$" type potentials $V$ by repeating  the arguments as in the proof of Theorem 1.3 in \cite{BGM}. 

\begin{thrm}\label{DF1}
Let $u$ solve 
\[
\La u + Vu =0
\]
in $B_1$ where $V$ satisfies \eqref{vassump1}. Then there exists $R_0$ universal such that for all $r \leq R_0$, we have
\begin{equation}\label{m}
||u||_{L^{\infty}(B_r)} \geq C_1 \left(\frac{r}{R_0}\right)^{C_2 (\sqrt{K}+1)},
\end{equation}
where $C_1, C_2$ have the same dependence as in Theorem \ref{main}. 
\end{thrm}
It is to be noted that  in the Euclidean case, the constant $C$ can be taken to be the $C^{1}$ norm of $V$. As previously mentioned,  Theorem \ref{DF1} has already been proven in \cite{BG1} differently     using a variant of the   frequency function approach which in turn is inspired by the work of Zhu in \cite{Zhu1}  for the standard Laplacian ( see also \cite{BG} for the extension of the frequency function approach of Zhu to variable coefficients at the boundary). 
 For a historical account, we  note that such an estimate in the Euclidean case was first established by Bakri in the above cited paper \cite{Bk} using the Euclidean version of the Carleman estimate \eqref{df}.  Such result generalised the sharp vanishing order estimate of Donnelly and Fefferman in \cite{DF1} for eigenfunctions of the Laplacian on compact manifolds.  Finally, we refer to the  paper \cite{Ru1} for an interesting generalisation to nonlocal equations of the quantitative uniqueness result in \cite{Bk}, and also to \cite{B} for a generalisation to Carnot groups of arbitrary step. 
 
 \medskip

We then    study strong unique continuation for sublinear  equations of the type
 \begin{equation}\label{sub}
 -\La u=  f((z,t), u) \psi + Vu, 
 \end{equation}
 where $V$ satisfies \eqref{vasump}, and $f$ and its primitive, $G((z,t),  s) = \int_{0}^{s} f((z,t) s) ds$, satisfies the following assumptions analogous to those in \cite{Ru} and \cite{SW} in the uniformly elliptic case:   
 \begin{equation}\label{a1}
 \begin{cases}
 f((z,t), 0) =0,
 \\
  0 < s f((z,t), s) \leq q G((z,t), s), \ \text{for some $q \in (1, 2)$ and  $s \in (-1, 1) \setminus \{0\}$},
  \\
  |\nabla_{(z,t)} f| \leq K |f|,\ |\nabla_{(z,t)} G| \leq K G,
  \\
  f((z,t), s) \leq \kappa s^{p-1}\ \text{for some $p \in (1,2)$},
  \\
  G((z,t), 1) \geq \ve_0\ \text{for some $\ve_0>0$}.
\end{cases}
\end{equation}
 We note that the conditions   in \eqref{a1} imply that  for some constant $c_0, c_1$, we have 
\begin{equation}\label{a0}
 c_1 s^{p} \geq  G(., s) \geq c_0 s^{q},\ \text{for}\ s \in (-1,1).
  \end{equation}
  A  prototypical $f$ satisfying \eqref{a1} is 
  \[
  f((z,t), u) = \sum_{i=1}^\ell c_i(z,t) |u|^{q_i-2} u,\ \text{where for each}\ i, q_i \in (1,2), 0<k_0<c_i< k_1, \text{and}\ |\nabla c_i| < K,
  \]
 for some $k_0, k_1\ \text{and}\ K$.  In this case, we can take  $q= \max \{q_i\}$ and $p= \min \{q_i \}$.

Unique continuation properties for uniformly elliptic nonlinear equations of the type 
\[
-\operatorname{div}(A(x) \nabla u) = f(x, u),
\]
with $f$ satisfying the assumptions in \eqref{a1}, have been recently studied in \cite{SW} and \cite{Ru}. More precisely,  weak unique continuation properties for such  sublinear equations 
were first obtained in \cite{SW} using the frequency function approach. Subsequently, in \cite{Ru} the author established the strong unique continuation property (see also \cite{ST} and \cite{To}). In this  work we generalise R\"uland's result to degenerate elliptic equations of the type \eqref{sub}. For related results in the parabolic setting, we refer to \cite{AB} and \cite{BMa}. 
 We have the following generalization of the result of R\"uland for the variable coefficient Baouendi-Grushin operators. 

\begin{thrm}\label{main2}
Let $u \in S^{2,2}(B_1)$ be a solution to \eqref{sub} in $B_1$ where $f$ satisfies the bounds in \eqref{a1} and $V$ satisfies \eqref{vasump}. Furthermore, assume that $||u||_{L^{\infty}(B_1)} \leq 1$.  If $u$ vanishes to infinite order at $(0,0)$ in the sense of Definition \ref{v0}, then $u \equiv 0$.

\end{thrm}

The reader should  note that the assumption on the sign of $f, G$ in \eqref{a1} is not restrictive because, even in the Euclidean case, the strong unique continuation property fails when $f= -|u|^{q-2}u$ and $A=\mathbb{I}$. This follows from a one dimensional  counterexample in \cite{SW}. 

Our final result concerns the following  non-Euclidean analogue   of a Carleman estimate due to Regbaoui in \cite{Reg}.

\begin{thrm} \label{hardy}
	There exists  universal $C>0$ such that for every $\beta >0$ sufficiently large,  $R_0$ sufficiently small and $u \in S^{2,2}_{0}(B_R \setminus \{0\})$ with  $\text{supp}\ u \subset (B_R \setminus \{0\})$ for $R\leq R_0$, one has
	\begin{align}\label{har1}
	&\beta^3 \int_{B_R} \rho^{-Q} e^{\beta (\log \rho)^2}u^2 \mu\ dzdt +  \beta \int_{B_R}  \rho^{-Q+2} e^{\beta (\log \rho)^2} \langle A \, Xu, Xu\rangle  dz dt \\
	&\leq C  \, \int \rho^{-Q+4} e^{\beta (\log \rho)^2} (\La  u)^2 \mu^{-1} \, dz dt. \notag
	\end{align}
	
\end{thrm}

As a corollary of the estimate \eqref{har1}, we obtain the following   unique continuation result  for a class of   scaling critical zero order  perturbations of $\La$ ( as in \eqref{hj} below)  by  an obvious modification of the arguments in \cite{BM}.

\begin{thrm}\label{hardy1}
Let $u \in S^{2,2}(B_1)$ with $|u| \leq 1$  be a solution to 
\begin{equation}\label{hj}
\La u = Vu,
\end{equation}
where the potential $V$ satisfies the following Hardy type  growth assumption, 
\begin{equation}\label{V1}
|V(z,t)| \leq C \frac{\psi}{\rho^2}.
\end{equation}
Assume that $u$ vanishes at $(0,0)$ in the following sense,
\begin{equation}\label{vp1}
\int_{B_r} u^2 \psi = O(e^{-k(\log r)^2}),\ \text{for all $k>0$.}
\end{equation}
Then $u \equiv 0$.
\end{thrm}
We note that Theorem \ref{hardy1}   can be regarded as a slight improvement of  Theorem 4.4 in \cite{G} where it is instead assumed  that $V= V^+ - V^{-}$ with
\[
0 \leq V^{+}(z,t) \leq C \frac{\psi}{\rho^2}\ \text{and}\ 0 \leq V^{-}(z,t) \leq \delta \frac{\psi}{\rho^2}
\]
where $\delta$ is  sufficiently small. Theorem \ref{hardy1} thus  gets rid of such a smallness assumption.   It is to be noted that Theorem \ref{hardy1} gives a unique continuation result for \eqref{hj} which is somewhat  weaker than strong unique continuation because the notion of vanishing in \eqref{vp1} is stronger than the notion of vanishing to infinite order  as in Definition \ref{v0}. As previously mentioned in the introduction, the strong unique continuation result for the case of the Hardy type potentials has only been proven  when $\gamma=k=1$  in \cite{BM}. The proof of the main Carleman estimate in \cite{BM} however relies on some  delicate  projection operator estimates established previously in \cite{GarShen}. Such estimates  are not available  for the general Baouendi-Grushin operators and as of now,   the strong unique continuation property for Hardy type potentials in this degenerate setting remains an open question. 

\section{Proof of the Main results}\label{mn}

\subsection*{Proof of Theorem \ref{main}}
The proof of Theorem \ref{main} is a consequence of the following Carleman estimate following which one can repeat the arguments as in the proof of Theorem 1.3 in \cite{BGM}.  In the Euclidean case, such a Carleman estimate was first established by Escauriaza and Vessella in \cite{EV}. 

\begin{thrm} \label{thm2}
	For every $\ve \in (0,1)$, there exists $C>0$ sufficiently large and $R_0>0$  sufficiently small  universal depending also on $\ve$  such that for every $\alpha >0$ sufficiently large( depending also on $\ve$) and $u \in S^{2,2}_{0}(B_R \setminus \{0\})$ with  $\text{supp} \, u \subset (B_R \setminus \{0\})$ for $R \leq R_0$, one has
	\begin{align}\label{est1}
	&\alpha^3 \int  \rho^{-2\alpha-4+\epsilon} e^{2\alpha \rho^{\epsilon}}u^2 \mu\  + \alpha \int  \rho^{-2\alpha-2+\epsilon} e^{2\alpha  \rho^{\epsilon}} \langle AXu, Xu \rangle  \\
	&\leq C\, \int \rho^{-2\alpha} e^{2\alpha  \rho^{\epsilon}} (\La u)^2 \mu^{-1}. \notag
	\end{align}
	
\end{thrm}

\begin{proof}
First by a limiting type argument, it suffices to establish the estimate when $u$ is smooth. Let  $v=\rho^{-\beta} e^{\alpha \rho^{\ve}}u$.   With such choice we have  $$u= \rho^{\beta} \, e^{-\alpha \rho^{\epsilon}} \, v,$$ where  $\beta$  is to be determined later ($\beta$ would finally depend on $\alpha$ and $Q$!). Then we have  that
\begin{eqnarray*}
\La u&=&X_i(a_{ij} X_ju)=X_i(a_{ij} X_j(\rho^{\beta} \, e^{-\alpha \rho^{\epsilon}} v))=X_i\left(a_{ij} \left[(\rho^{\beta} \, e^{-\alpha \rho^{\epsilon}}) \, X_jv+X_j(\rho^{\beta} \, e^{-\alpha \rho^{\epsilon}}) v\right]\right)\\
&&=X_i\left(a_{ij} (\rho^{\beta} \, e^{-\alpha \rho^{\epsilon}}) \, X_jv\right)+X_i\left(a_{ij}X_j(\rho^{\beta} \, e^{-\alpha \rho^{\epsilon}}) v\right)\\
&&=\La v (\rho^{\beta} \, e^{-\alpha \rho^{\epsilon}})+2 a_{ij} \, X_jv X_i(\rho^{\beta} \, e^{-\alpha \rho^{\epsilon}})+X_i(a_{ij}X_j(\rho^{\beta} \, e^{-\alpha \rho^{\epsilon}})) v\\
&&=\La v (\rho^{\beta} \, e^{-\alpha \rho^{\epsilon}})+2 a_{ij} \, X_jv X_i(\rho^{\beta} \, e^{-\alpha \rho^{\epsilon}})+\La (\rho^{\beta} \, e^{-\alpha \rho^{\epsilon}}) v
\end{eqnarray*}

Now using $$\La =\Ba+X_i(b_{ij} X_j),$$ we note that the above expression can be further rewritten as
\begin{eqnarray*}
	\La u=\La v (\rho^{\beta} \, e^{-\alpha \rho^{\epsilon}})+2 a_{ij} \, X_jv X_i(\rho^{\beta} \, e^{-\alpha \rho^{\epsilon}})+\Ba(\rho^{\beta} \, e^{-\alpha \rho^{\epsilon}}) v+ X_i(b_{ij} X_j \rho^{\beta} \, e^{-\alpha \rho^{\epsilon}}) v.
\end{eqnarray*}

%$$\Ba u=v\Ba (r^{\beta} \, e^{-\alpha r^{\epsilon}})+2\sum_{i=1}^N X_i (r^{\beta} \, e^{-\alpha r^{\epsilon}})X_iv  + \Ba v \, r^{\beta} \, e^{-\alpha r^{\epsilon}}.$$
Now we  calculate last two terms of the right hand side of the above equation. As in \cite{BGM}, we have 
\begin{align*}
\Ba(\rho^{\beta} \ e^{-\alpha \rho^{\ve}})= \left(\alpha^2 \ve^2 \rho^{\beta+2\ve-2} + \beta(\beta+Q-2)\rho^{\beta-2}- \alpha\ve\left((2\beta+\ve+ Q-2)\right) \rho^{\beta+\ve-2}\right) e^{-\alpha \rho^{\ve}} \psi,   
\notag
\end{align*}
and similarly we have 
\begin{eqnarray*}
X_i(\rho^{\beta} \, e^{-\alpha \rho^{\epsilon}})=\rho^{\beta-1} \, e^{-\alpha \rho^{\epsilon}}(\beta-\alpha \epsilon \rho^{\epsilon}) X_i\rho,
\end{eqnarray*}
and also
\begin{eqnarray*}
&&	X_iX_j(\rho^{\beta} \, e^{-\alpha \rho^{\epsilon}})=X_i\left[\rho^{\beta-1} \, e^{-\alpha \rho^{\epsilon}}(\beta-\alpha \epsilon \rho^{\epsilon}) X_j\rho\right]\\
&&	=\rho^{\beta-2} \, e^{-\alpha \rho^{\epsilon}}\left(\beta(\beta-1)-\alpha \epsilon (2\beta +\epsilon-1)+(\alpha \epsilon)^2 \rho^{2\epsilon}\right) X_i\rho X_j \rho+\rho^{\beta-1} \, e^{-\alpha \rho^{\epsilon}}(\beta-\alpha \epsilon \rho^{\epsilon}) X_i X_j\rho.
\end{eqnarray*} 
Thus it follows, 
\begin{eqnarray*}
	&&\La u=	\La v (\rho^{\beta} \, e^{-\alpha \rho^{\epsilon}})+2 a_{ij} \, X_jv X_i\rho \left[\rho^{\beta-1} \, e^{-\alpha \rho^{\epsilon}}(\beta-\alpha \epsilon \rho^{\epsilon})\right]\\
	&&+\left[ \left(\alpha^2 \ve^2 \rho^{\beta+2\ve-2} + \beta(\beta+Q-2)\rho^{\beta-2}- \alpha\ve\left((2\beta+\ve+ Q-2)\right) \rho^{\beta+\ve-2}\right) e^{-\alpha \rho^{\ve}} \psi\right] v\\
	&&+\left[X_ib_{ij} \cdot (\beta -\alpha \epsilon \rho^{\epsilon}) \rho^{\beta-1} e^{-\alpha \rho^{\epsilon}} X_j \rho\right] v+ \left[b_{ij}\left(\rho^{\beta-1} \, e^{-\alpha \rho^{\epsilon}}(\beta-\alpha \epsilon \rho^{\epsilon})\right)X_i X_j \rho\right]	v \\
&&+\left[b_{ij}[\rho^{\beta-2} e^{-\alpha \rho^{\epsilon}}(\beta(\beta-1)-\alpha \epsilon \rho^{\epsilon}(2\beta+\epsilon-1)+(\alpha \epsilon)^2 \rho^{2\epsilon})]X_i \rho X_j \rho\right] v.
\end{eqnarray*}

Now in terms of $F$ and $\mu$, we note that $\La u$ can be written in the following way, 
\begin{eqnarray}\label{b1}
&&\La u=	\La v (\rho^{\beta} \, e^{-\alpha \rho^{\epsilon}})+2 \mu Fv \left[\rho^{\beta-2} \, e^{-\alpha \rho^{\epsilon}}(\beta-\alpha \epsilon \rho^{\epsilon})\right]\\
&&+\left[ \left(\alpha^2 \ve^2 \rho^{\beta+2\ve-2} + \beta(\beta+Q-2)\rho^{\beta-2}- \alpha\ve\left((2\beta+\ve+ Q-2)\right) \rho^{\beta+\ve-2}\right) e^{-\alpha \rho^{\ve}} \psi\right] v \nonumber\\
&&+\left[X_ib_{ij} \cdot (\beta -\alpha \epsilon \rho^{\epsilon}) \rho^{\beta-1} e^{-\alpha \rho^{\epsilon}} X_j \rho\right] v+\left[a_{ij}\left(\rho^{\beta-1} \, e^{-\alpha \rho^{\epsilon}}(\beta-\alpha \epsilon \rho^{\epsilon})\right)X_i X_j \rho\right]	v \nonumber\\
&&+\left[b_{ij}[\rho^{\beta-2} e^{-\alpha \rho^{\epsilon}}(\beta(\beta-1)-\alpha \epsilon \rho^{\epsilon}(2\beta+\epsilon-1)+(\alpha \epsilon)^2 \rho^{2\epsilon})]X_i \rho X_j \rho\right] v. \nonumber 
\end{eqnarray}

Before proceeding further, we make the following discursive remark.

\begin{rmrk}
From now on  unless otherwise specified, we will be  following the Einstein notation for summation over repeated indices.
\end{rmrk}

Now using  $(a+b)^2\geq a^2+2ab$ with $a= 2 \beta \, \mu Fv \rho^{\beta-2} $ and with  $b$ being the  rest of the terms in \eqref{b1}  above,  we  obtain

\begin{align}\label{k}
&e^{2\alpha \rho^{\epsilon}}(\La u)^2 \geq 4\beta^2  \rho^{2\beta-4} \mu^2 (Fv)^2+4\beta \mu \rho^{2\beta-2}  \, Fv \La v -8 \alpha \beta \, \epsilon \mu^2 (Fv)^2 \rho^{2\beta-4+\epsilon}\\
&+4 \beta \, \mu  \rho^{\beta-2}\bigg[ \left[\left(\alpha^2 \ve^2 \rho^{\beta+2\ve-2} + \beta(\beta+Q-2)\rho^{\beta-2}- \alpha\ve\left((2\beta+\ve+ Q-2)\right) \rho^{\beta+\ve-2}\right)  \psi \right]\notag \\
&+\left[X_ib_{ij} \cdot (\beta -\alpha \epsilon \rho^{\epsilon}) \rho^{\beta-1} X_j \rho\right]+\left[b_{ij}\left(\rho^{\beta-1} \, (\beta-\alpha \epsilon \rho^{\epsilon})\right)X_i X_j \rho\right]\notag \\
&+\left[b_{ij}[\rho^{\beta-2} (\beta(\beta-1)-\alpha \epsilon \rho^{\epsilon}(2\beta+\epsilon-1)+(\alpha \epsilon)^2 \rho^{2\epsilon})]X_i \rho X_j \rho\right] \bigg] Fv \cdot	v \nonumber.
\end{align}
Hence from \eqref{k} we have, 
\begin{align}\label{rt1}
&\int \rho^{-2\alpha} e^{2\alpha \rho^{\epsilon}}(\La u)^2  \mu^{-1} \\
& \geq \int [4\beta^2  \rho^{-2\alpha+2\beta-4} \mu (Fv)^2 -8 \alpha \beta \, \epsilon\rho^{-2\alpha+2\beta-4+\epsilon} \mu (Fv)^2 ] \,   +\int 4\beta \rho^{-2\alpha+2\beta-2}  \, Fv \La v \notag\\
&
+ \int 4 \beta \,  \rho^{-2\alpha+\beta-2}\bigg[\left[\left(\alpha^2 \ve^2 \rho^{\beta+2\ve-2} + \beta(\beta+Q-2)\rho^{\beta-2}- \alpha\ve\left((2\beta+\ve+ Q-2)\right) \rho^{\beta+\ve-2}\right)  \psi \right] \notag\\
&+\left[X_ib_{ij} \cdot (\beta -\alpha \epsilon \rho^{\epsilon}) \rho^{\beta-1} X_j \rho\right]+\left[a_{ij}\left(\rho^{\beta-1} \, (\beta-\alpha \epsilon \rho^{\epsilon})\right)X_i X_j \rho\right] \notag\\
&+\left[b_{ij}[\rho^{\beta-2} (\beta(\beta-1)-\alpha \epsilon \rho^{\epsilon}(2\beta+\epsilon-1)+(\alpha \epsilon)^2 \rho^{2\epsilon})]X_i \rho X_j \rho\right]\bigg] Fv \cdot	v   \notag\\
&= \int [4\beta^2  \rho^{-2\alpha+2\beta-4} \mu (Fv)^2 -8 \alpha \beta \, \epsilon \mu (Fv)^2 \rho^{-2\alpha+2\beta-4+\epsilon}] \,  +\int 4\beta \rho^{-2\alpha+2\beta-2}  \, Fv \La v \,  \notag\\
&+ \int 4 \beta \,  \rho^{-2\alpha+2\beta-3} \left[X_ib_{ij} \cdot (\beta -\alpha \epsilon \rho^{\epsilon}) X_j \rho\right] Fv\cdot v,\notag\\
&
+ \int 4 \beta \,  \rho^{-2\alpha+\beta-2}\bigg[\left[\left(\alpha^2 \ve^2 \rho^{\beta+2\ve-2} + \beta(\beta+Q-2)\rho^{\beta-2}- \alpha\ve\left(2\beta+\ve+ Q-2t\right) \rho^{\beta+\ve-2}\right)  \psi \right]\notag\\
&+\left[b_{ij}[\rho^{\beta-2} (\beta(\beta-1)-\alpha \epsilon \rho^{\epsilon}(2\beta+\epsilon-1)+(\alpha \epsilon)^2 \rho^{2\epsilon})]X_i \rho X_j \rho\right] +\left[b_{ij}\left(\rho^{\beta-1} \, (\beta-\alpha \epsilon \rho^{\epsilon})\right)X_i X_j \rho\right]\bigg] F\left(\frac{v^2}{2}\right)\notag
\end{align}
The following integral in \eqref{rt1} above, i.e. 
\begin{align}\label{rta1}
&
 \int 4 \beta \,  \rho^{-2\alpha+\beta-2}\bigg[\left[\left(\alpha^2 \ve^2 \rho^{\beta+2\ve-2} + \beta(\beta+Q-2)\rho^{\beta-2}- \alpha\ve\left((2\beta+\ve+ Q-2)\right) \rho^{\beta+\ve-2}\right) \psi \right]\\
&+\left[b_{ij}[\rho^{\beta-2} (\beta(\beta-1)-\alpha \epsilon \rho^{\epsilon}(2\beta+\epsilon-1)+(\alpha \epsilon)^2 \rho^{2\epsilon})]X_i \rho X_j \rho\right] +\left[b_{ij}\left(\rho^{\beta-1} \, (\beta-\alpha \epsilon \rho^{\epsilon})\right)X_i X_j \rho\right]\bigg] F\left(\frac{v^2}{2}\right). \notag
\end{align}
is handled using integration by parts as follows. 
%
%we get   using \eqref{hg} and \eqref{Zpsi} that the following holds, 
%
%\begin{align}\label{rta1}
%&-\int \operatorname{div}\bigg(4 \beta \,  \rho^{-2\alpha+\beta-2}\bigg[\left[b_{ij}[\rho^{\beta-2} (\beta(\beta-1)-\alpha \epsilon \rho^{\epsilon}(2\beta+\epsilon-1)+(\alpha \epsilon)^2 \rho^{2\epsilon})]X_i \rho X_j \rho\right]  \nonumber \\
%&+\left[b_{ij}\left(\rho^{\beta-1} \, (\beta-\alpha \epsilon \rho^{\epsilon})\right)X_i X_j \rho\right]\bigg] (\frac{\rho}{\mu} a_{ij} X_j\rho)\bigg) (\frac{v^2}{2}) dz dt.
%\end{align}
%First we note that the vector field  $F$ can be represented in the following way,
%$$F=Z-\frac{\sigma}{\mu} Z+\frac{\rho}{\mu}\sum_{i,j=1}^Nb_{ij}X_i \rho X_j.$$

We  first observe from (i) in Theorem \ref{Est1} and also by using $F\rho=\rho$  that for  every $\gamma \geq 0$,
\begin{equation}\label{i1}
\operatorname{div}(\rho^{-Q+\gamma}F)=\gamma \rho^{-Q+\gamma}+O(1) \rho^{-Q+\gamma+1}.\end{equation}
Now we look at each individual term in  \eqref{rta1}.
We let 
\begin{eqnarray} \label{choice}
\beta = \frac{2\alpha+4 - Q}{2},
\end{eqnarray} which gives $2\beta - 2 \alpha - 4 = - Q$.  With such a choice,   it follows using Theorem \ref{Est1} i) and ii) that the following holds,  
\begin{equation}\label{p1}
2\beta^2  (\beta+Q-2) \int \rho^{2\beta-2\alpha-4} F(v^2) \psi  = 2 \beta^2( \beta+Q-2) \int \operatorname{div} (\rho^{-Q} F\psi) v^2 =  O(1) \beta^2(\beta+Q-2)  \int \rho^{-Q+1} v^2 \psi.
\end{equation}
Similarly we have, 
\begin{equation}\label{p2}
- 2 \beta \alpha \ve( 2 \beta +\ve + Q-2) \int \rho^{2\beta-2\alpha-4+\ve} F(v^2) \psi = 2 \beta \alpha \ve^2 (2 \beta + \ve+ Q-2) \left(  \int \rho^{-Q+\ve} v^2 \psi +O(1) \int \rho^{-Q+\ve+1} v^2 \psi\right),  
\end{equation}
and
\begin{equation}\label{p4}
2 \beta (\alpha \ve)^2 \int \rho^{2\beta-2\alpha-4+2\ve} F(v^2) \psi = -2 \beta \alpha^2 \ve^3  \left (\int \rho^{-Q+2\ve} v^2 \psi+O(1) \int \rho^{-Q+2\ve+1} v^2 \psi\right).  
\end{equation}
Since $\ve <1$, thus from \eqref{p1}-\eqref{p4} we deduce the following estimate for all  $R$ small enough depending also on $\ve$, 
\begin{align}\label{p5}
&\int 4 \beta   \rho^{-2\alpha + \beta-2}\bigg[ \left[\left(\alpha^2 \ve^2 \rho^{\beta+2\ve-2} + \beta(\beta+Q-2)\rho^{\beta-2}- \alpha\ve\left(2\beta+\ve+ Q-2\right) \rho^{\beta+\ve-2}\right)  \psi \right]Fv. v\\
& \geq  c \beta^3 \ve^2 \int \rho^{-Q+\ve} v^2 \psi, \notag
\end{align}
for some universal $c>0$.

Again by  integrating by parts and by  using $|F (b_{ij} X_i \rho X_j \rho)| \leq C\rho \psi$ and the Hypothesis \eqref{H}, we note    that the following holds, 
\begin{eqnarray*}
&&\int 4 \beta \rho^{-2\alpha+\beta -2} \left[b_{ij}[\rho^{\beta-2} (\beta(\beta-1)-\alpha \epsilon \rho^{\epsilon}(2\beta+\epsilon-1)+(\alpha \epsilon)^2 \rho^{2\epsilon})]X_i \rho X_j \rho\right]  F\left(\frac{v^2}{2}\right) \\
	&&= -\int \operatorname{div} \bigg(4\beta \rho^{-2\alpha+\beta-2} \left[b_{ij}[\rho^{\beta-2} (\beta(\beta-1)-\alpha \epsilon \rho^{\epsilon}(2\beta+\epsilon-1)+(\alpha \epsilon)^2 \rho^{2\epsilon}) ]X_i \rho X_j \rho\right]F\bigg) \frac{v^2}{2}\\
	&& \geq - C \beta^3 \int \rho^{-Q+1} v^2 \psi,
\end{eqnarray*}

for some universal $C$.  Likewise, we have that 
\begin{align}\label{p7}
&\int  4 \beta \rho^{-2\alpha+\beta-2} \left[b_{ij}\left(\rho^{\beta-1} \, (\beta-\alpha \epsilon \rho^{\epsilon})\right)X_i X_j \rho\right] F\left(\frac{v^2}{2}\right) 
\\
& = - \int 4 \beta \operatorname{div} ( \rho^{-2\alpha+\beta-2} \left[b_{ij}\left(\rho^{\beta-1} \, (\beta-\alpha \epsilon \rho^{\epsilon})\right)X_i X_j \rho\right] F) \frac{v^2}{2}\notag
\\
& \geq - C\beta^2 \int \rho^{-Q+1} v^2 \psi.
\notag
 \end{align}
 We note that in \eqref{p7} above, we also used the estimate from Lemma \ref{lma3.3} and also that  $|b_{ij} X_i X_j \rho| \leq C \psi$. Thus from \eqref{p5}-\eqref{p7} it follows that
 \begin{align}\label{p8}
& \int 4 \beta \,  \rho^{-2\alpha+\beta-2}\bigg[\left[\left(\alpha^2 \ve^2 \rho^{\beta+2\ve-2} + \beta(\beta+Q-2)\rho^{\beta-2}- \alpha\ve\left((2\beta+\ve+ Q-2)\right) \rho^{\beta+\ve-2}\right) \psi \right]\\
&+\left[b_{ij}[\rho^{\beta-2} (\beta(\beta-1)-\alpha \epsilon \rho^{\epsilon}(2\beta+\epsilon-1)+(\alpha \epsilon)^2 \rho^{2\epsilon})]X_i \rho X_j \rho\right] +\left[b_{ij}\left(\rho^{\beta-1} \, (\beta-\alpha \epsilon \rho^{\epsilon})\right)X_i X_j \rho\right]\bigg] F\left(\frac{v^2}{2}\right) \notag \\ & \geq c \beta^3 \ve^2 \int \rho^{-Q+\ve} v^2 \psi\notag, 
  \end{align}
  provided $R$ is  chosen small enough depending also on $\ve$.

Now using $\sum_{i,j=1}^N|X_i b_{ij} \, X_j \rho| \leq C\mu$,  we obtain by an application of  Cauchy-Schwartz inequality that the following holds
\begin{align}\label{p10}
&\int 4 \beta \,  \rho^{-2\alpha+2\beta-3} \left[X_ib_{ij} \cdot (\beta -\alpha \epsilon \rho^{\epsilon}) X_j \rho\right] Fv\cdot v \,  \\
&\leq \frac{1}{2}\left[\beta^2  \int \rho^{-2\alpha+2\beta-4} \mu (Fv)^2 dzdt + 8 \beta^2 \int \rho^{-2\alpha+2\beta -2} v^2 \mu \right]\notag\\
&=\frac{1}{2}\left[\beta^2 \int \rho^{-Q} \mu (Fv)^2  + 8 \beta^2 \int \rho^{-Q+2} v^2 \mu \right]\notag.
\end{align}
Again by using the Cauchy-Schwartz inequality, we note that if $R$ is chosen sufficiently small depending also on $\ve$, then for all large  enough $\beta$ we can ensure that 
\begin{align}\label{p11}
& \int [4\beta^2  \rho^{-2\alpha+2\beta-4} \mu (Fv)^2 -8 \alpha \beta \, \epsilon\rho^{-2\alpha+2\beta-4+\epsilon} \mu (Fv)^2 ]
\\
& \geq 3 \beta^2 \int \rho^{-Q} (Fv)^2 \mu.
\notag
\end{align}
%Thus from \eqref{p8}-\eqref{p11} we can deduce that for all $R$ small enough and $\beta$ large we have, 

We finally  estimate the remaining  integral in  \eqref{rt1}, i.e.
\[
\int 4\beta \rho^{-2\alpha+2\beta-2}  \, Fv \La v .\
\]
This is accomplished  using an appropriate Rellich type  identity similar to that for the constant coefficient case considered in \cite{BGM}. However in the present scenario of variable coefficients,  as the reader will see that an error term is incurred after the application of such an identity which is then eventually handled by an interpolation type argument.  We first note that with our choice of  $\alpha$ as in \eqref{choice},  such an integral equals
\[
\int 4\beta \rho^{-Q+2}  \, Fv \La v.
\]
  We now   state  the relevant Rellich type identity( see for instance Lemma 2.11 in \cite{GV}) which will be used:
\begin{align}\label{re}
& \int_{\partial B_R} \langle AXu,Xu\rangle \, \langle G,\nu \rangle  = 2 \int_{\partial B_R} a_{ij}X_i u \langle X_j,\nu \rangle Gu 
\\
& -  2\int_{B_R} a_{ij} (\operatorname{div} X_i) X_ju Gu  - 2 \int_{B_R} a_{ij} X_iu [X_j,G]u
\notag\\
& + \int_{B_R} \operatorname{div} G \langle AXu,Xu \rangle +  \int_{B_R} \langle (GA)Xu,Xu \rangle  - 2 \int_{B_R} Gu  X_i(a_{ij} X_j u),
\notag
\end{align}
where  $G$ is a vector field, $GA$ is the matrix with coefficients $Ga_{ij}$, $ \nu $ denotes  the outer unit normal to $B_r$, and the summation convention over repeated indices has been adopted.
Therefore with $G= \rho^{-Q+2}F,$ we obtain
\begin{align}\label{ok110}
& 4 \beta \int \rho^{-Q+2} Fv \La  v =2 \beta  \int \operatorname{div} [\rho^{-Q+2}F] \langle AXv, Xv \rangle  -4 \beta \int a_{ij} X_i v [X_i, \rho^{-Q+2} F] v
\\
&+2 \beta \int \langle (\rho^{-Q+2}F A) Xv, Xv \rangle \notag\\
& \geq  4\beta \int  \rho^{-Q+2} \langle AXv, Xv \rangle -C \beta \int \rho^{-Q+3} \langle AXv, Xv \rangle -4 \beta \int a_{ij} X_i v [X_i, \rho^{-Q+2} F] v.
\notag
\end{align}
In the  last step we used the fact that $\operatorname{div} [\rho^{-Q+2}F]=2\rho^{-Q+2}+O(1)\rho^{-Q+3}$ and $|Fa_{rs}|\leq C\rho.$
We next note  that 
\begin{eqnarray*}
&&[X_i,G]v = \rho^{-Q+2}[X_i, F]v+X_i[\rho^{-Q+2}]Fv\\
&&= \rho^{-Q+2}[X_i, F]v+(-Q+2) \rho^{-Q+1} X_i \rho \, Fv.
\end{eqnarray*}

This gives
\begin{align*}
a_{ij} X_i v [X_j,G] v & = (-Q+2) \rho^{-Q+1} \langle A X \rho,Xv \rangle Fv + \rho^{-Q+2} a_{ij} X_i v [X_j,F]v
\\
& = (-Q+2) \rho^{-Q} (Fv)^2 \mu   + \rho^{-Q+2} a_{ij} X_i v \left([X_j,F]v - X_j v\right)
\\
& \ \ \ + \rho^{-Q+2} \langle AXv,Xv\rangle,
\end{align*}
where we have used the fact that
\[
\rho \langle AX\rho,Xu\rangle = \mu Fu.
\]
Therefore using the estimate in Theorem \ref{Est1} vii) we obtain that for some universal $C$ the following holds,  

\begin{align}\label{ok120}
& 4 \beta \int \rho^{-Q+2} Fv \La v  \geq 4\beta \int  \rho^{-Q+2} \langle AXv, Xv \rangle \\
& - 4\beta \int  \rho^{-Q+2} \, \langle AXv,Xv\rangle 
 - C \beta   \int  \langle AXv,Xv\rangle \rho^{-Q+3} + 2\beta (Q-2) \int_{B_r} (Fv)^2 \rho^{-Q}  \mu.
\notag
\end{align}
Since $Q \geq 2$, therefore  from \eqref{ok120} we  deduce that the following holds,
\begin{equation} \label{rt3}
 4 \beta \int \rho^{-Q+2} Fv \La v +C \beta  \int \rho^{-Q+3}  \langle AXv, Xv \rangle  \geq 0.
\end{equation}

Therefore, by combining \eqref{p8}, \eqref{p10}, \eqref{p11}  and \eqref{rt3}, we finally obtain that 
\begin{eqnarray}\label{rt400}
&&\int \rho^{-2\alpha} \, e^{2\alpha \rho^{\epsilon}} \, (\La u )^2 \mu^{-1}  \\
&&\geq    3 \beta^2 \int  \rho^{-Q}  (Fv)^2  \mu  - C\beta  \int \rho^{-Q+3}  \langle AXv, Xv \rangle + c\beta^3 \ve^2 \int \rho^{-Q+\epsilon} v^2 \mu \nonumber
\end{eqnarray}
where $C$ and $c$ are universal constants. We now  estimate the term $$C\beta \int \rho^{-Q+3}  \langle AXv, Xv \rangle$$  in \eqref{rt400} using an interpolation type argument.  We first rewrite such an integral in terms of $\int \rho^{-2\alpha -1 } e^{2 \alpha \rho^{\epsilon}} <AXu, AXu>$ as follows, 
\begin{align}
	&\int \rho^{-2\alpha-1}  e^{2\alpha \rho^{\epsilon}} \langle AX u, X u \rangle = \int \rho^{-2\alpha-1}  e^{2\alpha \rho^{\epsilon}}\langle AX (\rho^{\beta}e^{-\alpha \rho^{\epsilon} } v),X (\rho^{\beta}e^{-\alpha r^{\epsilon} } v )\rangle\\
	&= \int \rho^{-2\alpha-1} 
	 e^{2\alpha \rho^{\epsilon}}\langle A \, (\beta-\alpha \epsilon \rho^{\epsilon}) \rho^{\beta-1} e^{-\alpha \rho^{\epsilon}} X\rho \, v+\rho^{\beta} e^{-\alpha \rho^{\epsilon}} A Xv, (\beta-\alpha \epsilon \rho^{\epsilon}) \rho^{\beta-1} e^{-\alpha \rho^{\epsilon}} X\rho \, v+\rho^{\beta} e^{-\alpha \rho^{\epsilon}}  Xv \rangle\notag \\
	&=\int \rho^{-2\alpha-1} e^{2\alpha \rho^{\epsilon}} \bigg[\mu \, \rho^{2\beta-2} e^{-2\alpha \rho^{\epsilon}} (\beta-\alpha \epsilon \rho^{\epsilon})^2  v^2+2 \mu Fv \, v (\rho^{2\beta-2} (\beta-\alpha \epsilon \rho^{\epsilon})e^{-2\alpha \rho^{\epsilon}})
	+ (e^{-2\alpha \rho^{\epsilon}} \rho^{2\beta} ) \langle A Xv, Xv\rangle \bigg]\notag\\
	&=\int \rho^{-Q+3}  \bigg[\frac{\mu \, (\beta-\alpha \epsilon \rho^{\epsilon})^2}{\rho^2} v^2+2 \mu \frac{(\beta-\alpha \epsilon \rho^{\epsilon})}{\rho^2} Fv \, v+  \,   \langle A Xv, Xv\rangle \bigg]\notag\\
	& = \int \rho^{-Q+3}  \bigg[\frac{\mu \,  (\beta-\alpha \epsilon \rho^{\epsilon})^2}{\rho^2} v^2- \mu \big[\rho^{-2}(\beta-\alpha \epsilon \rho^{\epsilon}) \operatorname{div}(F)
	+\rho^{-2}(\beta(-Q+1)+\alpha \epsilon (-Q+1+\epsilon))\big] v^2+  \,   \langle A Xv, Xv\rangle \bigg] \notag \\
	&= \int \rho^{-Q+3}  \bigg[\frac{\mu \,  (\beta-\alpha \epsilon \rho^{\epsilon})^2}{\rho^2} v^2- \mu \big[\rho^{-2}(\beta-\alpha \epsilon \rho^{\epsilon}) (Q+O(\rho))\notag\\
	&+\rho^{-2}(\beta(-Q+1)+\alpha \epsilon (-Q+1+\epsilon))\big] v^2+  \,   \langle A Xv, Xv\rangle \bigg] .\notag
\end{align}
Thus we get  for some universal $C_1$ that 
\begin{equation}\label{l1}
\int \rho^{-2\alpha-1} e^{2\alpha \rho^{\epsilon}}\langle A Xu, Xu\rangle  \geq  \int \rho^{-Q+3} \langle A Xv, Xv\rangle  - C_1 \beta\int \rho^{-Q+1} v^2.\end{equation}
Using \eqref{l1} in \eqref{rt400} we obtain that for all  large enough $\beta$ depending also on $\ve$  we  have, 
\begin{eqnarray}\label{rt41}
&& \hspace{.5in}\int \rho^{-2\alpha} \, e^{2\alpha \rho^{\epsilon}} \, (\La  u)^2 \mu^{-1} + C_1 \beta \int \rho^{-2\alpha-1} e^{2\alpha \rho^{\epsilon}}\langle A Xu, Xu\rangle \\
&&\geq    3 \beta^2 \int  \, \rho^{-Q}  (Fv)^2 \mu  +  c \beta^3 \ve^2  \int \rho^{-2\alpha-4+\epsilon} e^{2\alpha \rho^{\epsilon}}  u^2 \, \mu. \nonumber
\end{eqnarray}

%Hence 
%\begin{eqnarray}\label{rt5}
%&&\int_{B_R} \rho^{-Q+\epsilon} e^{\beta (\log \rho)^2} u^2 \psi +\beta \int_{B_R} \rho^{\alpha-2+\epsilon} e^{\beta (\log \rho)^2}|Xu|^2 \\
%&&\leq\frac{C}{ \, \epsilon} R^{\epsilon} \left[\frac{1}{\beta^2}\int_{B_R} e^{\beta (\log \rho)^2}\rho^{-2\alpha} \, (\Ba (u))^2 \psi^{-1}\right] + 0.99 \beta \int_{B_R} \rho^{\alpha-2+\epsilon} e^{\beta (\log \rho)^2}|Xu|^2.\nonumber
%\end{eqnarray}
 We now  show how to incorporate  the integral  $C\beta  \int_{B_R} \rho^{-2\alpha-2+\epsilon} e^{2\alpha \rho^{\epsilon}} \langle A Xu, Xu\rangle $ in the left hand side of \eqref{est1} by  interpolation. 

We have  by integration by parts,
\begin{align}\label{rt7}
& \int \rho^{-2\alpha-2+\epsilon} e^{2\alpha \rho^{\epsilon}} \langle A X u, X u \rangle  =\int \rho^{-2\alpha-2+\epsilon} e^{2\alpha \rho^{\epsilon}} \, \langle A X (\rho^{\beta} e^{-\alpha \rho^{\epsilon}} v), X (\rho^{\beta} e^{-\alpha \rho^{\epsilon}} v)\rangle\\
&=-\int <X(\rho^{-2\alpha-2+\epsilon} e^{2\alpha \rho^{\epsilon}}) , A \, X(\rho^{\beta} e^{-\alpha \rho^{\epsilon}} v) > [\rho^{\beta} e^{-\alpha \rho^{\epsilon}} v]-\int \La  (\rho^{\beta} e^{-\alpha \rho^{\epsilon}} v) \, v \, [\rho^{-2\alpha+\beta-2+\epsilon} e^{\alpha \rho^{\epsilon}}].
\notag
\end{align}
Now we look at each individual term on the right hand side of \eqref{rt7}. We have 
\begin{align} \label{ok4}
&-\int \langle X(\rho^{-2\alpha-2+\epsilon} e^{2\alpha \rho^{\epsilon}}) , AX(\rho^{\beta} e^{-\alpha \rho^{\epsilon}} v)\rangle \rho^{\beta} e^{-\alpha \rho^{\epsilon}} v=-\int e^{2\alpha \rho^{\epsilon}} \\
& [\rho^{-2\alpha-3+\epsilon}  \, (-2\alpha-2+\epsilon+2\alpha \epsilon \rho^{\epsilon})] \langle X \rho , A[\rho^{\beta} e^{-\alpha \rho^{\epsilon}} \, Xv+(\beta-\alpha \epsilon \rho^{\epsilon}) \, \rho^{\beta-1} \,e^{-\alpha \rho^{\epsilon}} \, X\rho \, v ] (\rho^{\beta} e^{-\alpha \rho^{\epsilon}} v)\rangle \notag\\
&=-\int \rho^{-2\alpha+2\beta-4+\epsilon} (-2\alpha-2+\epsilon+2\alpha \epsilon \rho^{\epsilon}) [(\beta-\alpha \epsilon \rho^{\epsilon}) \mu  \, v^2+\mu \, Fv \cdot v] \, dzdt
\notag
\\
&=\int \rho^{-Q+\epsilon} (2\alpha+2-\epsilon-2\alpha \epsilon \rho^{\epsilon}) (\beta-\alpha \epsilon \rho^{\epsilon}) \mu  \, v^2 \notag\\
&+ \int \rho^{-Q+\epsilon} (2\alpha+2-\epsilon-2\alpha \epsilon \rho^{\epsilon}) \mu \, Fv\cdot v.
\notag
\end{align}
Let $c_0 =\frac{c}{100}$ where $c$ is as in \eqref{rt41}.  From \eqref{ok4} it follows that
\begin{align}\label{ok10}
& c_0 \beta \ve^2  \int \rho^{-2\alpha-2+\epsilon} e^{2\alpha \rho^{\epsilon}} \langle A X u, X u \rangle \leq  c_0 \beta \ve^2 \int \rho^{-Q+\epsilon} (2\alpha+2-\epsilon-2\alpha \epsilon \rho^{\epsilon}) (\beta-\alpha \epsilon \rho^{\epsilon}) \mu  \, v^2 dzdt \\
& + c_0 \beta \ve^2  \bigg| \int \rho^{-Q+\epsilon} (2\alpha+2-\epsilon-2\alpha \epsilon \rho^{\epsilon}) \mu \, Fv\cdot v \,\bigg| 
+ c_0 \beta \ve^2  \bigg| \int \La  (\rho^{\beta} e^{-\alpha \rho^{\epsilon}} v) \, v \, [\rho^{-2\alpha+\beta-2+\epsilon} e^{\alpha \rho^{\epsilon}}]\bigg|. \notag
\end{align}
Now by  applying Cauchy-Schwartz  inequality to the integrals 
\[
c_0 \beta  \ve^2 \bigg|\int \rho^{-Q+\epsilon} (2\alpha+2-\epsilon-2\alpha \epsilon \rho^{\epsilon}) \mu \, Fv\cdot v \,\bigg| \
\text{and}\ c_0 \beta \ve^2  \bigg| \int \La  (\rho^{\beta} e^{-\alpha \rho^{\epsilon}} v) \, v \, [\rho^{-2\alpha+\beta-2+\epsilon} e^{\alpha \rho^{\epsilon}}]\bigg|,\]
we obtain from \eqref{ok10} that the following inequality holds for all large enough $\beta$,
\begin{align}\label{ok11}
& c_0 \beta \ve^2  \int \rho^{-2\alpha-2+\epsilon} e^{2\alpha \rho^{\epsilon}} \langle A X u, X u \rangle \leq  4 c_0 \beta^3 \ve^2 \int \rho^{-Q+\ve}  v^2 \mu 
\\
&  + 4 c_0 \beta^2  \ve^2 \int \rho^{-Q} (Fv)^2 \mu  + c_0  \ve^2   \int \rho^{-2 \alpha} e^{2\alpha \rho^{\epsilon}} (\La u)^2 \mu^{-1}
\notag
\end{align}
Using  \eqref{rt41} into \eqref{ok11} above, we obtain
\begin{align}\label{o13}
& c_0 \beta \ve^2  \int \rho^{-2\alpha-2+\epsilon} e^{2\alpha \rho^{\epsilon}} \langle A X u, X u \rangle 
\\
& \leq C   \int \rho^{-2 \alpha} e^{2\alpha \rho^{\epsilon}} (\La u)^2 \mu^{-1} + C\beta \int \rho^{-2\alpha -1 }  \langle AXu, Xu \rangle.
\notag 
\end{align}
Now since $\ve < 1$, therefore  if $R$ is chosen small enough depending also on $\ve$, then the following integral in \eqref{o13}, i.e.
\[
C\beta \int \rho^{-2\alpha -1 }  \langle AXu, Xu \rangle,\]
can be absorbed in the left hand side of \eqref{o13} since $\ve<1$ and  we thus  deduce that the following holds for some new $c_0, C$,
\begin{align}\label{y1}
& c_0 \beta \ve^2  \int \rho^{-2\alpha-2+\epsilon} e^{2\alpha \rho^{\epsilon}} \langle A X u, X u \rangle 
\\
& \leq C   \int \rho^{-2 \alpha} e^{2\alpha \rho^{\epsilon}} (\La u)^2 \mu^{-1}.\notag
\end{align}
The estimate \eqref{est1} now follows  by using \eqref{y1} in  \eqref{rt41}.

\end{proof}

\subsection*{Proof of Theorem \ref{DF}}

\begin{proof}
 As before, we let $u= \rho^{\beta} e^{\alpha \rho^{\epsilon}}v$ where $\alpha$ and $\beta$ are related as in \eqref{choice}.  In terms of $v$, we  have that
	\begin{align}\label{f111}
		\La u + Vu  &= \La v (\rho^{\beta} \, e^{-\alpha \rho^{\epsilon}})+2 \mu Fv \left[\rho^{\beta-2} \, e^{-\alpha \rho^{\epsilon}}(\beta-\alpha \epsilon \rho^{\epsilon})\right]+\left[X_ia_{ij} \cdot (\beta -\alpha \epsilon \rho^{\epsilon}) \rho^{\beta-1} e^{-\alpha \rho^{\epsilon}} X_j \rho\right] v\\
		&+\left[a_{ij}[\rho^{\beta-2} e^{-\alpha \rho^{\epsilon}}(\beta(\beta-1)-\alpha \epsilon \rho^{\epsilon}(2\beta+\epsilon-1)+(\alpha \epsilon)^2 \rho^{2\epsilon})]X_i \rho X_j \rho\right] v \nonumber \\
		&+\left[a_{ij}\left(\rho^{\beta-1} \, e^{-\alpha \rho^{\epsilon}}(\beta-\alpha \epsilon \rho^{\epsilon})\right)X_i X_j \rho\right]	v + V \rho^\beta \, e^{-\alpha \rho^{\epsilon}} \, v. \nonumber
	\end{align}
	
	Now again the integral 
	\[
	\int \rho^{-2\alpha} \, e^{2\alpha  \rho^{\epsilon}} \, (\La u + V u)^2 \mu^{-1} 
	\]
	is estimated from below  by using $(a+b)^2 \geq a^2 + 2ab$, with $a= 2\beta \rho^{\beta-2} \mu Fv$ and $b$ being the rest of the terms in \eqref{f111}.  Arguing as in the proof of Theorem \ref{thm2} we obtain,
	
	\begin{eqnarray}\label{rt4}
&& \hspace{.5in}\int \rho^{-2\alpha} \, e^{2\alpha \rho^{\epsilon}} \, (\La  u + Vu)^2 \mu^{-1}  + C_1 \beta \int \rho^{-2\alpha-1} e^{2\alpha \rho^{\epsilon}}\langle A Xu, Xu\rangle \\
&&\geq    3 \beta^2 \int \, \rho^{-Q}  (Fv)^2 \mu  +  c \beta^3 \ve^2  \int \rho^{-2\alpha-4+\epsilon} e^{2\alpha \rho^{\epsilon}}  u^2 \, \mu  + 4 \beta \int \rho^{-2\alpha +2\beta - 2} Fv V v. \nonumber
\end{eqnarray}	
	We note that the additional integral in \eqref{rt4} is incurred   due to  the presence of the additional  term $V \rho^\beta e^{-\alpha \rho^{\epsilon}} v $  in \eqref{f111} ( that is not present in  \eqref{b1} !). Such an integral is estimated as follows. We have
	\begin{align}\label{ad11}
	&	4 \beta \int \rho^{-2\alpha + 2\beta -2}  Fv \, V  \, v=2 \beta \int V \,\rho^{-Q + 2}  F(v^2) \notag\\
	&=-2 \beta \int FV \,\rho^{-Q + 2} \, v^2 +4 \beta \int V \,\rho^{-Q + 2}  v^2. 
	\end{align}
Recall $V$  satisfying the bound
\begin{equation*}
	|V(z,t)| \leq K \psi \text{ and }  	|ZV(z,t)| \leq K \psi.
\end{equation*}
Thus both the integrals in \eqref{ad11} can be controlled by the following term in \eqref{rt4}, i.e.
\[
c \beta^3 \ve^2  \int \rho^{-2\alpha-4+\epsilon} e^{2\alpha \rho^{\epsilon}}  u^2\mu, \]
provided 
\[
\frac{c \beta^3 \ve^2}{2} >    C K \beta \]   which  in turn  can be ensured by choosing 
\begin{equation}\label{ch1}
\alpha \geq  C_1 (K^{1/2} +1)\end{equation}  ( in view of \eqref{choice}) where $C_1$ is some universal constant.

	Finally, we show how to incorporate  the integral  $ \beta \int \rho^{-2\alpha-2+\epsilon} e^{2\alpha \rho^{\epsilon}} \langle A Xu, Xu\rangle dxdt $ in the left hand side of \eqref{df}  by an interpolation type argument as before.
	
	We have  by integration by parts,
\begin{align}\label{ty1}
&  \beta\int \rho^{-2\alpha-2+\epsilon} e^{2\alpha \rho^{\epsilon}} \langle A X u, X u \rangle  =  \beta\int \rho^{-2\alpha-2+\epsilon} e^{2\alpha \rho^{\epsilon}} \, \langle A X (\rho^{\beta} e^{-\alpha \rho^{\epsilon}} v), X (\rho^{\beta} e^{-\alpha \rho^{\epsilon}} v)\rangle
\\		
&= - \beta \int X(\rho^{-2\alpha-2+\epsilon} e^{2\alpha \rho^{\epsilon}}) \cdot A \, X(\rho^{\beta} e^{-\alpha \rho^{\epsilon}} v) [\rho^{\beta} e^{-\alpha \rho^{\epsilon}} v]-  \beta \int \La  (\rho^{\beta} e^{-\alpha \rho^{\epsilon}} v) [\rho^{-2\alpha-2+\epsilon} e^{2\alpha \rho^{\epsilon}}] \, [\rho^{\beta} e^{-\alpha \rho^{\epsilon}} v]
\notag
\end{align}

Now by writing $\La u = (\La u + V u) - Vu$, we obtain  from \eqref{ty1} that the following inequality holds,
\begin{align}\label{ty2}
& \beta\int \rho^{-2\alpha-2+\epsilon} e^{2\alpha \rho^{\epsilon}} \langle A X u, X u \rangle  
\\
&\leq  \beta  \int \rho^{-Q+\epsilon} (2\alpha+2-\epsilon-2\alpha \epsilon \rho^{\epsilon}) (\beta-\alpha \epsilon \rho^{\epsilon}) \mu  \, v^2    + \beta   \bigg| \int \rho^{-Q+\epsilon} (2\alpha+2-\epsilon-2\alpha \epsilon \rho^{\epsilon}) \mu \, Fv\cdot v \,\bigg| \notag
 \\
& +  \beta  \bigg| \int ( \La u  + Vu) \, v \, [\rho^{-2\alpha+\beta-2+\epsilon} e^{\alpha \rho^{\epsilon}}]\bigg|  + C \beta^3  \int  \rho^{-Q+2 +\ve} v^2 \mu,
\notag
\end{align}
where in the last inequality above, we used that $|V| \leq \beta^2 \psi$ which follows from \eqref{ch1}. Therefore, at  this point by suitably applying Cauchy Schwartz inequality to the integrals,
\[
 \bigg| \int \rho^{-Q+\epsilon} (2\alpha+2-\epsilon-2\alpha \epsilon \rho^{\epsilon}) \mu \, Fv\cdot v \,\bigg| 
 \]
 and
 \[
 \beta  \bigg| \int ( \La u  + Vu) \, v \, [\rho^{-2\alpha+\beta-2+\epsilon} e^{\alpha \rho^{\epsilon}}]\bigg|, 
 \]
 we can argue as in the proof of Theorem \ref{thm2}  using the estimate \eqref{rt4} instead of \eqref{rt41} to get to the desired conclusion.

\end{proof}

\subsection*{Proof of Theorem \ref{main2}}

%We now turn our attention to the  proof of Theorem \ref{main4}. 
The proof of Theorem \ref{main2} is a consequence of the following Carleman estimate after  which one can repeat the arguments in \cite{BGM}. 

\begin{thrm}\label{sub1}
	Let $\ve \in (0,1)$, $1< q< 2$ and let $f$ satisfy the assumptions in \eqref{a1}. Then for  every $\ve>0$, there exists $C$ universal depending also on $\ve$ and the constants in \eqref{a1}  such that for every $\alpha >0$ sufficiently large and $u \in S^{2,2}_{0}(B_R \setminus \{0\})$ with  $\text{supp}\ u \subset (B_R \setminus \{0\})$, one has
	\begin{align}\label{f10}
&	\alpha^3 \int e^{2\alpha  \rho^{\epsilon}} \left[ \rho^{-2\alpha-4+\epsilon}u^2 \mu  \rho^{-2\alpha-2} |u|^q \,  \mu \right] \\
&+ \alpha  \int   \rho^{-2\alpha-2+\epsilon} e^{2\alpha  \rho^{\epsilon}} \langle AXu, Xu \rangle   \leq C  \, \int \rho^{-2\alpha} \, e^{2\alpha  \rho^{\epsilon}} \, (\La  u + f((z,t), u) \psi)^2 \mu^{-1},.\notag
	\end{align}
	provided $R \leq R_0$ where $R_0$ is sufficiently small.
	\end{thrm}

\begin{proof}
	The proof is similar to that of Theorem \ref{thm2} except that we additionally exploit the intrinsic nature of the sublinearity $f((z,t), u)$ and the structural assumptions  in \eqref{a1}.  As before, we let $u= \rho^{\beta} e^{\alpha \rho^{\epsilon}}v$ where $\alpha$ and $\beta$ are related as in \eqref{choice}.  In terms of $v$, we  have that
	\begin{align}\label{f11}
	\La u + f((z,t), u) \psi &= \La v (\rho^{\beta} \, e^{-\alpha \rho^{\epsilon}})+2 \mu Fv \left[\rho^{\beta-2} \, e^{-\alpha \rho^{\epsilon}}(\beta-\alpha \epsilon \rho^{\epsilon})\right]+\left[X_ia_{ij} \cdot (\beta -\alpha \epsilon \rho^{\epsilon}) \rho^{\beta-1} e^{-\alpha \rho^{\epsilon}} X_j \rho\right] v\\
	&+\left[a_{ij}[\rho^{\beta-2} e^{-\alpha \rho^{\epsilon}}(\beta(\beta-1)-\alpha \epsilon \rho^{\epsilon}(2\beta+\epsilon-1)+(\alpha \epsilon)^2 \rho^{2\epsilon})]X_i \rho X_j \rho\right] v \nonumber \\
	&+\left[a_{ij}\left(\rho^{\beta-1} \, e^{-\alpha \rho^{\epsilon}}(\beta-\alpha \epsilon \rho^{\epsilon})\right)X_i X_j \rho\right]	v + f((z,t), \rho^\beta \, e^{-\alpha \rho^{\epsilon}} \, v)  \psi. \nonumber
	\end{align}
	
	Now again the integral 
	\[
	\int \rho^{-2\alpha} \, e^{2\alpha  \rho^{\epsilon}} \, (\La u + f((z,t), u) \psi)^2 \mu^{-1}
	\]
	is estimated from below  by using $(a+b)^2 \geq a^2 + 2ab$, with $a= 2\beta \rho^{\beta-2} \mu Fv$ and $b$ being the rest of the terms in \eqref{f11}. In this case,  we note that all the other terms  ( with the exception of \eqref{ad} below) are handled in the same way as before and  therefore we only need to  focus our attention on the following additional term  which is incurred due to  the presence of the additional  term $f((z,t) \rho^\beta e^{-\alpha \rho^{\epsilon}} v)  \psi$  in \eqref{f11} ( that is not present in  \eqref{b1} !), i.e.
	\begin{equation}\label{ad}
	4 \beta \int \rho^{-2\alpha + \beta -2} e^{\alpha \rho^{\epsilon}} Fv f((z,t), \rho^\beta e^{-\alpha \rho^{\epsilon}} \, v) \psi.
	\end{equation}
	Now from the  fact that $G$ is the $"s-$antiderivative" of $f$ we have
	\begin{align}\label{i19}
	& F \left(G((z,t), \rho^{\beta} e^{-\alpha \rho^{\epsilon}}v) \right)= Fv f( (z,t), \rho^\beta e^{-\alpha \rho^{\epsilon}} v) \rho^{\beta} e^{-\alpha \rho^{\epsilon}}+ (\beta-\alpha \epsilon \rho^{\epsilon}) \rho^{\beta} \, e^{-\alpha \rho^{\epsilon}} v f((z,t), \rho^{\beta} \, e^{-\alpha \rho^{\epsilon}} v)
	\\
	& + \langle \nabla_{(z,t)} G((z,t), \rho^{\beta} e^{-\alpha \rho^{\epsilon}} v), F\rangle.
	\notag
	\end{align}
	Note that  in \eqref{i19} above,  we used the fact that $F\rho^\beta=\beta \rho^{\beta}$. Then by using \eqref{i19} we obtain
	\begin{align}\label{m10}
	&4 \beta \int \rho^{-2\alpha + \beta -2} e^{\alpha \rho^{\epsilon}}  Fv f((z,t), \rho^\beta e^{-\alpha \rho^{\epsilon}} v) \psi = 4 \beta \int F \left(G((z,t), \rho^{\beta} e^{-\alpha \rho^{\epsilon}} v) \right) \rho^{-2\alpha-2} e^{\alpha \rho^{\epsilon}} \psi
	\\
	& - 4\beta \int (\beta-\alpha \epsilon \rho^{\epsilon})\rho^{-2\alpha -2} f((z,t), \rho^{\beta} \, e^{-\alpha \rho^{\epsilon}} v) \rho^{\beta} e^{\alpha \rho^{\epsilon}}v \psi  -4 \beta \int \rho^{-2\alpha-2} e^{2\alpha \rho^{\epsilon}} \langle \nabla_{(z,t)} G((z,t), \rho^{\beta} e^{-\alpha \rho^{\epsilon}}v), F\rangle  \psi.
	\notag
	\end{align}
	Now from  the third condition in \eqref{a1} we have 
	\begin{equation}\label{i8}
	f((z,t), \rho^{\beta} e^{-\alpha \rho^{\epsilon}}v) \rho^{\beta} e^{-\alpha \rho^{\epsilon}}v \leq q G((z,t), \rho^{\beta} e^{-\alpha \rho^{\epsilon}} v)
	\end{equation}
	and the fourth condition in \eqref{a1} implies 
	\begin{equation}\label{i7}
	\bigg<\nabla_{z, t} G, F\bigg > \leq  C_2 G.
	\end{equation}
	Thus by  using \eqref{i8} and \eqref{i7} in \eqref{m10} we get the following inequality,  
	\begin{align}\label{m11}
	&4 \beta  \int \rho^{-2\alpha + \beta -2}  e^{\alpha \rho^{\epsilon}} Fv f((z,t), \rho^\beta \,  e^{-\alpha \rho^{\epsilon}}  v)  \, \psi
	\\
	& \geq 4 \beta \int F \bigg(G((z,t), \rho^{\beta} e^{-\alpha \rho^{\epsilon}}v)\bigg) \rho^{-2\alpha-2} e^{2\alpha \rho^{\epsilon}} \psi- 4\beta q  \int (\beta-\alpha \ve  \rho^{\epsilon})\rho^{-2\alpha -2} e^{2\alpha \rho^{\epsilon}} G((z,t), \rho^\beta e^{-\alpha \rho^{\epsilon}} v) \psi
	\notag
	\\
	& - 4 C_2  \beta \int \rho^{-2\alpha-2} e^{2\alpha \rho^{\epsilon}} G((z,t), \rho^{\beta} e^{-\alpha \rho^{\epsilon}} v)  \psi.
	\notag
	\end{align}
	where in order to estimate the last integral, we used that $\mu \sim \psi$. 
	Now the first term in the right hand side of \eqref{m11}, i.e. the integral
	\[
	4 \beta \int F \bigg( G((z,t), \rho^{\beta} e^{-\alpha \rho^{\epsilon}}v)  \bigg) \rho^{-2\alpha-2} e^{2\alpha \rho^{\epsilon}} \psi
	\]
	is handled using integration by parts in the following way using the estimates in Theorem \ref{Est1} i) and ii). 
	
	\begin{align}\label{m13}
	&4 \beta \int F\bigg( G((z,t), \rho^{\beta} e^{-\alpha \rho^{\epsilon}}v) \bigg) \rho^{-2\alpha-2} e^{2\alpha \rho^{\epsilon}}\psi
	\\
	& = - 4 \beta \int G((z,t), \rho^{\beta} e^{-\alpha \rho^{\epsilon}}v) \operatorname{div}(\rho^{-2\alpha-2} e^{2\alpha \rho^{\epsilon}} F \psi ) 
	\notag
	\\
	&=  \int (8\beta (\beta-1)  - 8 \alpha \beta \ve \rho^{\ve}   -4\beta O(\rho))\rho^{-2\alpha-2} e^{2\alpha \rho^{\epsilon}} G((z,t), \rho^{\beta} v) \psi.
	\notag
	\end{align}
	
	We note that over here we used \eqref{choice} which implies that
	\begin{eqnarray*}
	\operatorname{div}(\rho^{-2\alpha-2} e^{2\alpha \rho^{\epsilon}} F)&=& [-2\alpha-2+2\alpha \epsilon  \rho^{\epsilon}]\rho^{-2\alpha-2} e^{2\alpha \rho^{\epsilon}}+\rho^{-2\alpha-2} e^{2\alpha \rho^{\epsilon}} \operatorname{div} F\\
	&&=(-2 (\beta -1)+2\alpha \epsilon \rho^{\epsilon} +O(\rho))\rho^{-2\alpha-2} e^{2\alpha \rho^{\epsilon}}.
	\end{eqnarray*}
	Now since $q <2 $, by using \eqref{m11} and \eqref{m13}   we obtain that 
	\begin{align} \label{3.37}
	&\int \rho^{-2\alpha + \beta -2} e^{\alpha \rho^{\epsilon}} 4 \beta Fv f((z,t), \rho^\beta e^{-\alpha \rho^{\epsilon}} v) \psi \\
	& \geq   \int ( 4\beta^2( 2-q)+4\alpha \beta \ve  \rho^{\epsilon}(q-2)-C \rho -4 C_2 \beta - 8 \beta)  \rho^{-2\alpha-2} e^{2\alpha \rho^{\epsilon}} G((z,t), \rho^{\beta} e^{-\alpha \rho^{\epsilon}} v) \psi \notag
	\\
	& \geq    c \beta^2  \int \rho^{-2\alpha-2} e^{2\alpha \rho^{\epsilon}} G((z,t), \rho^{\beta} e^{-\alpha \rho^{\epsilon}} v) \psi,\ \text{for large enough $\beta$ provided $R$ is small enough}.
	\notag
	\end{align}
	
	Thus it follows from the computations  as in the proof of Theorem \ref{thm2} and by using \eqref{3.37}  that the following inequality holds,
	\begin{eqnarray}\label{rt40}
&& \hspace{.5in}\int \rho^{-2\alpha} \, e^{2\alpha \rho^{\epsilon}} \, (\La  u + f((z,t), u) \psi )^2 \mu^{-1} + C_1 \beta \int \rho^{-2\alpha-1} e^{2\alpha \rho^{\epsilon}}\langle A Xu, Xu\rangle \\
&&\geq    3 \beta^2 \int \, \rho^{-Q}  (Fv)^2 \mu  +  c \beta^3 \ve^2  \int \rho^{-2\alpha-4+\epsilon} e^{2\alpha \rho^{\epsilon}}  u^2 \, \mu  + c \beta^2  \int \rho^{-2\alpha-2} e^{2\alpha \rho^{\epsilon}} G((z,t), u) \psi. \nonumber
\end{eqnarray}

Finally as before, we show how to incorporate  the integral  $ \beta  \int \rho^{-2\alpha-2+\epsilon} e^{2\alpha \rho^{\epsilon}} \langle A Xu, Xu\rangle$ in the left hand side of \eqref{f10} by interpolation. 

We have,
\begin{align}\label{rt71}
& \beta \int\rho^{-2\alpha-2+\epsilon} e^{2\alpha \rho^{\epsilon}} \langle A X u, X u \rangle =\beta \int \rho^{-2\alpha-2+\epsilon} e^{2\alpha \rho^{\epsilon}} \, \langle A X (\rho^{\beta} e^{-\alpha \rho^{\epsilon}} v), X (\rho^{\beta} e^{-\alpha \rho^{\epsilon}} v)\rangle\\
&=-\beta \int X(\rho^{-2\alpha-2+\epsilon} e^{2\alpha \rho^{\epsilon}}) \cdot A \, X(\rho^{\beta} e^{-\alpha \rho^{\epsilon}} v) [\rho^{\beta} e^{-\alpha \rho^{\epsilon}} v]-\beta \int \La u \, v \, [\rho^{-2\alpha+\beta-2+\epsilon} e^{\alpha \rho^{\epsilon}}].
\notag
\end{align}
Now by rewriting $\La u= (\La u + f((z, t), u) \psi ) - f((z,t). u) \psi$ we obtain from \eqref{rt71} that
\begin{align}\label{bn}
& \beta \int\rho^{-2\alpha-2+\epsilon} e^{2\alpha \rho^{\epsilon}} \langle A X u, X u \rangle \\
&\leq  \beta  \int \rho^{-Q+\epsilon} (2\alpha+2-\epsilon-2\alpha \epsilon \rho^{\epsilon}) (\beta-\alpha \epsilon \rho^{\epsilon}) \mu  \, v^2     + \beta   \bigg| \int \rho^{-Q+\epsilon} (2\alpha+2-\epsilon-2\alpha \epsilon \rho^{\epsilon}) \mu \, Fv\cdot v \,\bigg| \notag
 \\
& +  \beta  \bigg| \int ( \La u  + f((z,t), u) ) \, v \, [\rho^{-2\alpha+\beta-2+\epsilon} e^{\alpha \rho^{\epsilon}}]\bigg|  +  \beta   \int  \rho^{-2\alpha-2 +\ve} e^{2\alpha \rho^\ve} f((z, t), u) u \mu
\notag
\\
& \leq  \beta  \int \rho^{-Q+\epsilon} (2\alpha+2-\epsilon-2\alpha \epsilon \rho^{\epsilon}) (\beta-\alpha \epsilon \rho^{\epsilon}) \mu  \, v^2     + \beta   \bigg| \int \rho^{-Q+\epsilon} (2\alpha+2-\epsilon-2\alpha \epsilon \rho^{\epsilon}) \mu \, Fv\cdot v \,\bigg| \notag
 \\
& +  \beta  \bigg| \int ( \La u  + f((z,t), u) ) \, v \, [\rho^{-2\alpha+\beta-2+\epsilon} e^{\alpha \rho^{\epsilon}}]\bigg|  +  \beta   q \int  \rho^{-2\alpha-2 +\ve} e^{2\alpha \rho^\ve} G((z,t), u)  \psi,
\notag
\end{align}
where in the last inequality in \eqref{bn} above, we used that $u f((z,t), u) \leq q G((z, t), u)$.  We now note that the last integral in \eqref{bn} above, i.e.
\[
\beta   q \int  \rho^{-2\alpha-2 +\ve} e^{2\alpha \rho^\ve} G((z,t), u)  \mu
\]
can be estimated from above by the following integral in \eqref{rt40}, i.e.
 \[c \beta^2  \int \rho^{-2\alpha-2} e^{2\alpha \rho^{\epsilon}} G((z,t), u) \psi,\]  provided $\beta$ is sufficiently large.  At this point , the rest of the argument  is similar  to that for Theorem \ref{thm2} where we use the inequality \eqref{rt40} instead of \eqref{rt41} and we finally arrive at the following estimate
 \begin{align}\label{f0}
&	\alpha^3 \int e^{2\alpha  \rho^{\epsilon}} \left[ \rho^{-2\alpha-4+\epsilon}u^2 \mu + \rho^{-2\alpha-2} G((z,t), u)  \,  \mu \right] \\
&+ \alpha  \int   \rho^{-2\alpha-2+\epsilon} e^{2\alpha  \rho^{\epsilon}} \langle AXu, Xu \rangle   \leq C \, \int \rho^{-2\alpha} \, e^{2\alpha  \rho^{\epsilon}} \, (\La  u + f((z,t), u) \psi)^2 \mu^{-1} .\notag
	\end{align} 
	The desired inequality \eqref{f10} now follows by using \eqref{a0} in \eqref{f0}.

\end{proof}

\subsection*{Proof of Theorem \ref{hardy}}

\begin{proof}
	Let  $v= e^{\frac{\beta}{2} (\log \rho)^{2}}   \, u$.   Then it follows, 	\begin{eqnarray*}
		\La u=\La v (e^{-\frac{\beta}{2} (\log \rho)^{2}} )+2 a_{ij} \, X_jv X_i(e^{-\frac{\beta}{2} (\log \rho)^{2}} )+\La (e^{-\frac{\beta}{2} (\log \rho)^{2}} ) v
	\end{eqnarray*}
Since $a_{ij}=\delta_{ij}+b_{ij}$, therefore we have, 
	
	\begin{eqnarray*}
		\La u=\La v (e^{-\frac{\beta}{2} (\log \rho)^{2}})+2 a_{ij} \, X_jv X_i(e^{-\frac{\beta}{2} (\log \rho)^{2}})+\Ba(e^{-\frac{\beta}{2} (\log \rho)^{2}}) v+ X_i(b_{ij} X_j (e^{-\frac{\beta}{2} (\log \rho)^{2}})) v.
	\end{eqnarray*}
	
	%$$\Ba u=v\Ba (r^{\beta} \, e^{-\alpha r^{\epsilon}})+2\sum_{i=1}^N X_i (r^{\beta} \, e^{-\alpha r^{\epsilon}})X_iv  + \Ba v \, r^{\beta} \, e^{-\alpha r^{\epsilon}}.$$
	Now, we compute the  last two terms in the right hand side of the above expression. By a standard  calculation we obtain, 	
	\begin{eqnarray*}
		\Ba(e^{-\frac{\beta}{2}(\log \rho)^2})=\psi \, e^{-\frac{\beta}{2}(\log \rho)^2} \, \rho^{-2}\left((\beta \log \rho)^2-\beta-(Q-2) \beta \log \rho \right),
	\end{eqnarray*}
	and
	\begin{eqnarray*}
		X_i(e^{-\frac{\beta}{2}(\log \rho)^2})=-\beta (\log \rho) \rho^{-1} \, e^{-\frac{\beta}{2}(\log \rho)^2} X_i\rho.
	\end{eqnarray*}
Similarly we have,  	 
	\begin{eqnarray*}
		&&	X_iX_j(e^{-\frac{\beta}{2}(\log \rho)^2})=X_i\left[-\beta (\log \rho) \rho^{-1} \, e^{-\frac{\beta}{2}(\log \rho)^2} X_j\rho\right]\\
		&&	=\beta \rho^{-2} \,  e^{-\frac{\beta}{2}(\log \rho)^2} \, \left(-1+(\log \rho)+\beta (\log \rho)^2\right) X_i\rho X_j \rho-\beta (\log \rho) \rho^{-1} \, e^{-\frac{\beta}{2}(\log \rho)^2} X_i X_j\rho.
	\end{eqnarray*} 
	Thus, 
	\begin{eqnarray*}
		&&\La u=	\La v (e^{-\frac{\beta}{2}(\log \rho)^2})+2 a_{ij} \, X_jv X_i\rho \left[-\beta (\log \rho) \rho^{-1} \, e^{-\frac{\beta}{2}(\log \rho)^2}\right]\\
		&&+\left[ \psi \, e^{-\frac{\beta}{2}(\log \rho)^2} \, \rho^{-2}\left((\beta \log \rho)^2-\beta-(Q-2) \beta \log \rho \right)\right] v\\
		&&- \left[X_ib_{ij}  \beta (\log \rho) \rho^{-1} \, e^{-\frac{\beta}{2}(\log \rho)^2} X_j\rho\right] v+ \left[b_{ij}\left(-\beta (\log \rho) \rho^{-1} \, e^{-\frac{\beta}{2}(\log \rho)^2}\right)X_i X_j \rho\right]	v \\
		&&+\left[b_{ij}\left(\beta \rho^{-2} \,  e^{-\frac{\beta}{2}(\log \rho)^2} \, \left(-1+(\log \rho)+\beta (\log \rho)^2\right)\right)X_i \rho X_j \rho\right] v.
	\end{eqnarray*}
%	We next collect several preliminary results that will be important in our proof. We first consider the quantity
%	\begin{eqnarray} \label{defmu}
%	\mu=\langle A X \rho, X\rho\rangle.
%	\end{eqnarray}
%	In view of the uniform ellipticity of $A$, we have
%	
%	\begin{eqnarray}
%	\lambda \psi \leq \mu \leq \lambda^{-1} \, \psi.
%	\end{eqnarray}
%	The following vector field $F$ will play an important role ine the paper:
%	\begin{eqnarray} \label{defF}
%	F=\frac{\rho}{\mu} \sum_{i,j=1}^Na_{ij} X_i \rho X_j.
%	\end{eqnarray}
%	We note that 
%	\begin{eqnarray} \label{3.5}
%	F\rho= \rho.
%	\end{eqnarray}
Consequently  in terms of the vector field $F$, we observe that $\La u $ can be written as, 
	
	\begin{eqnarray} \label{c1}
		&&\La u=	\La v (e^{-\frac{\beta}{2}(\log \rho)^2})+2 \mu Fv \left[-\beta (\log \rho) \rho^{-2} \, e^{-\frac{\beta}{2}(\log \rho)^2}\right]\\
		&&+\left[ \psi \, e^{-\frac{\beta}{2}(\log \rho)^2} \, \rho^{-2}\left((\beta \log \rho)^2-\beta-(Q-2) \beta \log \rho \right)\right] v\nonumber\\
		&&+\left[X_ib_{ij} \cdot \beta (-\log \rho) \rho^{-1} \, e^{-\frac{\beta}{2}(\log \rho)^2} X_j\rho\right] v+ \left[b_{ij}\left(-\beta (\log \rho) e^{-\frac{\beta}{2} (\log \rho)^2}\rho^{-1} \, \right)X_i X_j \rho\right]	v \nonumber\\
		&&+\left[b_{ij}\left(\beta \rho^{-2} e^{-\frac{\beta}{2} (\log \rho)^2}  \, \left(-1+(\log \rho)+\beta (\log \rho)^2\right)\right)X_i \rho X_j \rho\right] v.\nonumber
	\end{eqnarray}		
	Now using  $(a+b)^2\geq a^2+2ab$ with $a= 2 \beta \, (-\log \rho) \, \mu Fv \rho^{-2} $ and with  $b$ being the  rest of the terms on \eqref{c1}  above,  we  obtain,
	\begin{eqnarray*}
		e^{\beta(\log \rho)^2}(\La u)^2 &\geq& 4\beta^2  \rho^{-4} (\log \rho)^2 \,  \mu^2 (Fv)^2+4\beta \mu \rho^{-2} (-\log \rho)  \, Fv \La v \\
		&&+4 \beta \, (-\log \rho) \, \mu  \rho^{-2}\bigg[ \left[\psi \,  \, \rho^{-2}\left((\beta \log \rho)^2-\beta-(Q-2) \beta \log \rho \right) \right] \\
		&&+\left[X_ib_{ij} \cdot \beta (-\log \rho) \rho^{-1} X_j\rho\right]+\left[b_{ij}\left(\beta (-\log \rho) \rho^{-1} \right)X_i X_j \rho\right] \\
		&&+\left[b_{ij}\left(\beta \rho^{-2} \,   \, \left(-1+(\log \rho)+\beta (\log \rho)^2\right)\right)X_i \rho X_j \rho\right] \bigg] Fv \cdot	v \nonumber.
	\end{eqnarray*}
	Hence,
	\begin{align}\label{c2}
	&\int \rho^{-Q+4}  \, e^{\beta(\log \rho)^2}(\La
	 u)^2  \mu^{-1} \\
	& \geq \int [4\beta^2  \rho^{-Q} (\log \rho)^2 \, \mu (Fv)^2] \,  dz dt +\int 4\beta \rho^{-Q+2} (-\log \rho) \, Fv \La v  \notag\\
	&
	+ \int 4 \beta \,  \rho^{-Q+2} \, (-\log \rho) \bigg[\left[\psi  \, \rho^{-2}\left((\beta \log \rho)^2-\beta-(Q-2) \beta \log \rho \right)\right] \notag\\
	&+\left[X_ib_{ij} \cdot \beta (-\log \rho) \rho^{-1} \,  X_j\rho\right]+\left[b_{ij}\left(\beta (-\log \rho) \rho^{-1} \, \right)X_i X_j \rho\right] \notag\\
	&+\left[b_{ij}\left(\beta \rho^{-2}  \, \left(-1+(\log \rho)+\beta (\log \rho)^2\right)\right)X_i \rho X_j \rho\right]\bigg] Fv \cdot	v  \notag\\
	& = \int [4\beta^2  \rho^{-Q} (\log \rho)^2 \, \mu (Fv)^2] \,  dz dt +\int 4\beta \rho^{-Q+2} (-\log \rho) \, Fv \La v \notag\\
	&
	+ \int 4 \beta \,  \rho^{-Q+2} \, (-\log \rho) \bigg[\left[\psi \,  \, \rho^{-2}\left((\beta \log \rho)^2-\beta-(Q-2) \beta \log \rho \right)\right] \notag\\
	&+\left[X_ib_{ij} \cdot \beta (-\log \rho) \rho^{-1} \,  X_j\rho\right]+\left[b_{ij}\left(\beta (-\log \rho) \rho^{-1} \, \right)X_i X_j \rho\right] \notag\\
	&+\left[b_{ij}\left(\beta \rho^{-2} \,   \, \left(-1+(\log \rho)+\beta (\log \rho)^2\right)\right)X_i \rho X_j \rho\right]\bigg] F\left(\frac{v^2}{2}\right)\notag
	\end{align}
	We first handle the following term in the right hand side of  \eqref{c2} above,
	\begin{align}\label{c3}
	&
	\int 4 \beta \,  \rho^{-Q+2} \, (-\log \rho) \bigg[\left[\psi \,  \, \rho^{-2}\left((\beta \log \rho)^2-\beta-(Q-2) \beta \log \rho \right) \right]+\left[b_{ij}\left(\beta (-\log \rho) \rho^{-1} \, \right)X_i X_j \rho\right]\\
	& +\left[b_{ij}\left(\beta \rho^{-2} \,   \, \left(-1+(\log \rho)+\beta (\log \rho)^2\right)\right)X_i \rho X_j \rho\right]\bigg] F\bigg(\frac{v^2}{2}\bigg). \notag
	\end{align}
	
%	We note that vector vield $F$ can be represented in the following way
%	$$F=Z-\frac{\sigma}{\mu} Z+\frac{\rho}{\mu}\sum_{i,j=1}^Nb_{ij}X_i \rho X_j.$$
%	\begin{rmrk}\label{c4}
%		We observe that when $A(z,t) \equiv I_N$, then $B(z,t) \equiv 0_N$. In such case we see that $F \equiv Z$.
%	\end{rmrk}
	
 We now  look at each individual term in  \eqref{c3}.
	 %	\operatorname{div}(\rho^{-Q} \psi v^2 Z) = \rho^{-Q} Z(v^2) \psi.
%	\]
%	We thus find
%	\begin{equation}\label{zero}
%	\int \rho^{\alpha-4} Z(v^2) \psi = 0.
%	\end{equation}
%	
%	=================

First we observe that by applying  integration by parts to the integral
\[
\int 4 \beta \, [(\beta \log \rho)^2-\beta-(Q-2) \beta \log \rho] (-\log \rho) \rho^{-Q}\psi F \bigg(\frac{v^2}{2}\bigg)
\]

we get that the following holds, 

\begin{align} \label{c5}
& \int 4 \beta \, [(\beta \log \rho)^2-\beta-(Q-2) \beta \log \rho] (-\log \rho) \rho^{-Q}\psi F \bigg(\frac{v^2}{2}\bigg)=  4 \int \beta^3 \operatorname{div}[(\log \rho)^3\rho^{-Q} F\psi] \frac{v^2}{2}  \\
&-4 \int \beta^2 \operatorname{div}((\log \rho)\rho^{-Q} F\psi) \frac{v^2}{2} -4 \int \beta^2(Q-2) \operatorname{div}((\log \rho)^2\rho^{-Q} F\psi) \frac{v^2}{2} \notag
\end{align}

Now we estimate  each individual term in the right hand side of \eqref{c5}. We have  using the estimates in  Theorem \ref{Est1}, 

\begin{eqnarray}\label{kj1}
	4 \int \beta^3 \operatorname{div}[(\log \rho)^3\rho^{-Q} F\psi ] \frac{v^2}{2} \geq 6 \, \beta^3 \int \rho^{-Q}( \log \rho)^2 v^2 \psi - C\beta^3 \int \rho^{-Q+1}( \log \rho)^3 v^2 \psi.
\end{eqnarray}

Likewise it follows that, 
\begin{eqnarray*}
	&&-4 \int \beta^2 \operatorname{div}((\log \rho)\rho^{-Q} F\psi) \frac{v^2}{2} -4 \int \beta^2(Q-2) \operatorname{div}((\log \rho)^2\rho^{-Q} F\psi) \frac{v^2}{2} \\
	&& \geq2\beta^2	 \int  \rho^{-Q} [2(Q-2)(-\log \rho)-1]  \, v^2 \psi - C \int  \rho^{-Q+1}  [\beta^2(-\log \rho) + 4 \beta^2(Q-2)(\log \rho)^2 ].
\end{eqnarray*}	
	
Next, we see that

\begin{align}
	& 4 \beta^2 \int \rho^{-Q} (-\log \rho) (-1+(\log \rho)+\beta (\log \rho)^2)b_{ij} X_i \rho X_j \rho  F \bigg(\frac{v^2}{2} \bigg)\\ & = - 2 \beta^2 \int \operatorname{div} ( \rho^{-Q} (-\log \rho) (-1+(\log \rho)+\beta (\log \rho)^2)b_{ij} X_i \rho X_j \rho  F) v^2\notag 	\\
	& \geq - C\beta^2 \int \rho^{-Q+2} (-\log \rho) (-1+(\log \rho)+\beta (\log \rho)^2) v^2  \mu  - C\beta^2 \int \rho^{-Q+1} v^2 \mu - C\beta^3 \int \rho^{-Q+1} (\log \rho)^2 v^2 \mu
	\notag
		\end{align}
	
	Finally, the second integral in \eqref{c3}  can be estimated  using  the estimates in Theorem \ref{Est1} as well as the third derivative estimate in  Lemma \ref{lma3.3} in the following way, 
	\begin{align}\label{bv}
	& 4 \beta \int  \rho^{-Q+1} \left[ b_{ij} (-\log \rho)  X_iX_j \rho \right] F\bigg(\frac{v^2}{2} \bigg) \\
		&\int \operatorname{div}\bigg(4 \beta \,  \rho^{-Q+1}\left[b_{ij}\left(\beta (\log \rho)  \, \right)X_i X_j \rho F\right] \bigg)\frac{v^2}{2}\notag\\
		& \geq - C \beta^2 \int \rho^{-Q+1} (-\log \rho) v^2 \mu.\notag
			\end{align}
	Thus from \eqref{c5}-\eqref{bv} it follows that for all $R$ small enough, we have
	\begin{align} \label{c6}
	&\int 4 \beta \,  \rho^{-Q+2} \, (-\log \rho) \bigg[\left[\psi \,  \, \rho^{-2}\left((\beta \log \rho)^2-\beta-(Q-2) \beta \log \rho \right) \right]+\left[b_{ij}\left(\beta (-\log \rho) \rho^{-1} \, \right)X_i X_j \rho\right]\\
	& +\left[b_{ij}\left(\beta \rho^{-2} \,   \, \left(-1+(\log \rho)+\beta (\log \rho)^2\right)\right)X_i \rho X_j \rho\right]\bigg] F\bigg(\frac{v^2}{2}\bigg) \notag\\	
	& \geq 5 \, \beta^3 \int \rho^{-Q}( \log \rho)^2 v^2 \mu.\notag
	\end{align}
	Now using  $\sum_{i,j=1}^N|X_i b_{ij} \, X_j \rho| \leq C\mu$, we obtain by applying  Cauchy-Schwartz inequality that the following holds, 
	\begin{align}\label{re1}
		&\int 4 \beta \,  \rho^{-Q+1} \left[X_ib_{ij} \cdot \beta (-\log \rho) X_j \rho\right] Fv\cdot v 
		\\
		& \geq -  \beta^2 \int \rho^{-Q+1} ( Fv)^2 (\log(\rho))^2 \mu  - C \int \rho^{-Q+1} v^2 \mu\notag
	\end{align}

	We now estimate  the second integral in  \eqref{c2}, i.e.
	\[
	\int 4\beta \rho^{-Q+2} \, (-\log \rho)  \, Fv \La v.
	\]
	Now  in order to estimate  this integral, we use the Rellich type identity as in \eqref{re} with $G= \rho^{-Q+2} (-\log \rho) F$.
		 It follows  using \eqref{re}, the estimates in Theorem \ref{Est1} and by computations which are analogous to that in \eqref{ok110}-\eqref{rt3} that the following holds, 
	\begin{align}\label{c7}
	& 4 \beta \int \rho^{-Q+2} \, (-\log \rho) \, Fv \La  v  \geq  4\beta   \int \rho^{-Q}  ((Q-2)  (-\log \rho) +1)   (Fv)^2 \mu - 2  \beta \int \rho^{-Q+2}   <AXv, Xv>\\ &  - C\beta \int\rho^{-Q+3} (-\log \rho) \langle AXv, Xv \rangle
	 \geq 4\beta   \int \rho^{-Q}  ((Q-2)  (-\log \rho) +1)   (Fv)^2 \mu - \frac{5}{2} \beta \int \rho^{-Q+2}   <AXv, Xv>, \notag 		\end{align}
		
		where in the last inequality above,  we used that for all small enough $\rho$,
		\[
		C\rho^{-Q+3} \log (-\rho) \leq \frac{1}{2} \rho^{-Q+2}.
		\]

	Therefore by combining \eqref{c2}, \eqref{c6}, \eqref{re1} and \eqref{c7}, we finally deduce the following inequality for all $\beta$ large and $R$ small, 
	\begin{eqnarray}\label{c11}
	&&\int \rho^{-Q+4} \, e^{\beta (\log \rho)^{2}} \, (\La u)^2 \mu^{-1}  +\frac{5}{2} \beta \int  \rho^{-Q+2} \langle AXv, Xv \rangle \\
	&&\geq    3 \beta^2 \int  \rho^{-Q}(\log \rho)^2   (Fv)^2 \mu   + 4  \beta^3 \int \rho^{-Q} (\log \rho)^2 \,  v^2 \mu. \nonumber
	\end{eqnarray}
	
We now rewrite the integral $$\int  \rho^{-Q+2} \langle AXv, Xv \rangle$$ as follows.
	We have, 
	\begin{eqnarray*}
		&&\int \rho^{-Q+2} e^{\beta (\log \rho)^2} \langle AXu, Xu \rangle= \int \rho^{-Q+2} e^{\beta (\log \rho)^2}\langle A X e^{-\beta/2 (\log \rho)^2} v,X e^{-\beta/2 (\log \rho)^2} v \rangle\\
		&&=\int \rho^{-Q+2}  \bigg[\frac{(\beta \log \rho)^2}{\rho^2} \, \mu v^2+2\frac{(-\beta \log \rho)}{\rho^2} v Fv \mu+  \, <AXv, Xv> \bigg]\\
		&& \geq  \int \rho^{-Q+2 }  \bigg[\frac{(\beta \log \rho)^2}{\rho^2} \mu v^2+ \frac{\beta}{\rho^2} v^2 - C\beta \rho^{-1} (-\log \rho) v^2 \mu    +  \langle AX v, Xv \rangle\bigg].       \\
		&& \geq  \int \rho^{-Q+2}  \bigg[\frac{(\beta \log \rho)^2}{\rho^2} \mu v^2  +  \, \langle AX v, Xv \rangle\bigg]\ \text{(provided $R$ is small enough)}.
	\end{eqnarray*}
	Thus, we get 
	\begin{equation}\label{dom1}
	\int \rho^{-Q+2 } e^{\beta (\log \rho)^2}\langle AXu, Xu \rangle \geq  \int \rho^{-Q+2} \langle AX v, Xv \rangle + \beta^2 \int \rho^{-Q} (\log \rho)^2 v^2 \mu .  \end{equation}Using \eqref{dom1}   in \eqref{c11} we obtain, 
	\begin{eqnarray}\label{c12}
	&&\int \rho^{-Q+4} \, e^{\beta (\log \rho)^{2}} \, (\La (e^{-\frac{\beta}{2} (\log \rho)^{2}} v))^2 \mu^{-1} + \frac{5}{2} \beta \, \int \rho^{-Q+2 } e^{\beta (\log \rho)^2}\langle AXu, Xu \rangle \\
	&&\geq    3 \beta^2 \int_{B_R} \, \rho^{-Q} (\log \rho)^2  (Fv)^2  \mu +  \frac{13}{2} \beta^3 \int \rho^{-Q} (\log \rho)^2 \,  v^2 \mu. \nonumber
	\end{eqnarray}

	Finally, we show how to incorporate  the integral  $  \beta \int \rho^{-Q+2} e^{\beta (\log \rho)^{2}}  \langle A Xu, Xu\rangle $ in the left hand side of \eqref{har1} by an interpolation type argument as before.  We have,
	
\begin{align}\label{c15}
	&\beta  \int \rho^{-Q+2} e^{\beta (\log \rho)^2} (-\log \rho)^{\ve}  \langle A Xu, Xu\rangle 
	=- \beta \int \bigg< X(\rho^{-Q+2} e^{\beta (\log \rho)^2} ,  AX(e^{-\beta/2 (\log \rho)^2}v) \bigg> e^{-\beta/2 (\log \rho)^2}v \\ & -\beta \int \La(e^{-\beta/2 (\log \rho)^2}v) \rho^{-Q+2} e^{\beta/2 (\log \rho)^2}v.
	\notag
	\\
	& \leq - \beta \int \bigg< X(\rho^{-Q+2} e^{\beta (\log \rho)^2}),  AX(e^{-\beta/2 (\log \rho)^2}v) \bigg> e^{-\beta/2 (\log \rho)^2}v\notag\\ & + C \int \rho^{-Q+4}  e^{\beta (\log \rho)^2} (\La u)^2 \mu^{-1} + C\beta^2 \int \rho^{-Q}   v^2 
	\notag
	\\
	& \leq - \beta \int \bigg< X(\rho^{-Q+2} e^{\beta (\log \rho)^2},  AX(e^{-\beta/2 (\log \rho)^2}v) \bigg> e^{-\beta/2 (\log \rho)^2}v + C \int \rho^{-Q+4}  e^{\beta (\log \rho)^2} (\La u)^2 \mu^{-1} \notag\\
	&
	+ C \int \rho^{-Q+2} e^{\beta (\log \rho)^2} <A  Xu, Xu>, \notag
			\end{align}
			where in the last inequality in \eqref{c15} above, we used the estimate \eqref{c12}. 
			
			Now the following term in \eqref{c15} above, i.e. 
			\[
			- \beta \int \bigg< X(\rho^{-Q+2} e^{\beta (\log \rho)^2}  ),  AX(e^{-\beta/2 (\log \rho)^2}v) \bigg> e^{-\beta/2 (\log \rho)^2}v\]
			is estimated as follows. 
	
We have,

	\begin{align}\label{c19}
	&-\beta \int \langle X (\rho^{-Q+2} e^{\beta (\log \rho)^2}  ),  AX(e^{-\beta/2 (\log \rho)^2}v) \rangle e^{-\beta/2 (\log \rho)^2}v
	\\
	&=-\beta (-Q+2) \int \rho^{-Q}  \left[\mu \beta(-\log \rho) v^2+ \mu\, Fv \cdot v\right]    +\beta \int \rho^{-Q} 2\beta (-\log \rho)\left[(-\beta \log \rho) v^2\mu+  Fv \cdot v \mu \right]
	\notag
	\\
	&\leq \frac{5}{2}  \beta^3 \int \rho^{-Q} (\log \rho)^2  v^2 \mu + C \beta \int \rho^{-Q} (\log \rho)^2 (Fv)^2 \mu\ \text{(for all large $\beta$ and $R$ small)}
	\notag
	\\
	& \leq  C \int \rho^{-Q+4} e^{\beta (\log \rho)^2} (\La u)^2 \mu^{-1} + \left( \frac{25}{26} \beta  +C \right) \int \rho^{-Q+2}  \langle AXu, Xu \rangle,\notag
	\end{align}
	where in the last inequality above, we again used  the estimate \eqref{c12}. Thus from \eqref{c15} and \eqref{c19} we obtain 
	
	\begin{align}\label{grad}
	& \beta  \int \rho^{-Q+2} e^{\beta (\log \rho)^2}  \langle A Xu, Xu\rangle \leq 	C \int \rho^{-Q+4} e^{\beta (\log \rho)^2} (\La u)^2 \mu^{-1} + \left( C +\frac{25}{26} \beta\right) \int \rho^{-Q+2}  \langle AXu, Xu \rangle.
	\end{align}
	Now for all $\beta$ large  enough, we observe that the  following term in \eqref{grad} above, i.e.
	\[
	\left( C +\frac{25}{26} \beta \right) \int \rho^{-Q+2}  \langle AXu, Xu \rangle.\] can be absorbed in the left hand side of \eqref{grad} and we thus infer that the following estimate holds,
	
	\begin{align}\label{grad1}
	& \beta  \int \rho^{-Q+2} e^{\beta (\log \rho)^2}  \langle A Xu, Xu\rangle \leq 	C \int \rho^{-Q+4} e^{\beta (\log \rho)^2} (\La u)^2 \mu^{-1}	.
	\end{align}
	The desired estimate \eqref{har1} now follows from \eqref{c12} and \eqref{grad1}. 
		
\end{proof}

%\newpage

\section{Appendix}\label{ap}
\begin{proof}[Proof of Lemma \ref{lma3.3}]
 First note that \[
 F(b_{ij}X_i X_j \rho)=F(b_{ij}) \, X_iX_j \rho+(\frac{\rho}{\mu}\sum a_{qr} X_q\rho) \, b_{ij}  X_r(X_iX_j \rho).\]
 
 Now a standard tedious computation which uses the estimates in Lemma \ref{gauge}, Proposition \ref{prop3.4} and the hypothesis \eqref{H} shows that 
 \begin{equation}\label{cal1}
 \sum |(F(b_{ij}) X_i X_j \rho|  \leq C \psi.
 \end{equation}
 Consequently, we turn our attention to estimating the  term $(\frac{\rho}{\mu}\sum a_{qr} X_q\rho) \, b_{ij}  X_r(X_iX_j \rho).$ To do this, we need to compute the third derivatives of $\rho$. For that, we use the expressions for the second derivatives of $\rho$ as listed  in Lemma  \ref{secondder}. We first recall the expression for the derivatives of $\psi$ as in the proof of  Proposition 3.2 in \cite{GV}.  

	\begin{eqnarray*}
		X_l\psi=\begin{cases}
			2 \gamma \psi \frac{z_l}{|z|^2}-2\gamma \psi^2\frac{z_l}{\rho^2}, & \text{ for } 1\leq l\leq m\,\\
			-2\gamma(\gamma+1) \, \psi \, \frac{t_{l-m}|z|^{\gamma}}{\rho^{2\gamma+2}}, & \text{ for } m+1\leq l\leq N.
		\end{cases}
	\end{eqnarray*}
	We split our consideration into the following cases.
\begin{enumerate}
	\item  For $1\leq r \leq m$ and $1\leq i,j \leq m,$ we have:
	\begin{align*}
	&X_r(X_iX_j \rho)=-(2\gamma+1) X_r(z_iz_j) \frac{\psi^2}{\rho^3}-(2\gamma+1) z_iz_jz_r \frac{\psi^2}{\rho^3}\left[\frac{4\gamma}{|z|^2}-\frac{\psi}{\rho^2}(4\gamma+3)\right]\\
		&+X_r\left(2\gamma \frac{z_iz_j}{|z|^2}+\delta_{ij}\right) \frac{\psi}{\rho}+\left(2\gamma \frac{z_iz_j}{|z|^2}+\delta_{ij}\right) \, \frac{\psi}{\rho}z_r\left[\frac{2\gamma}{|z|^2}-\frac{\psi}{\rho^2}(2\gamma +1)\right]\\
		&=-(2\gamma+1) \left(z_j \delta_{ri}+z_i \delta_{rj}\right) \frac{\psi^2}{\rho^3}-(2\gamma+1) z_iz_jz_r \frac{\psi^2}{\rho^3}\left[\frac{4\gamma}{|z|^2}-\frac{\psi}{\rho^2}(4\gamma+3)\right]\\
		&+2\gamma \left(\frac{z_i\delta_{rj}+z_j \delta_{ri}}{|z|^2}-2\frac{z_iz_jz_r}{|z|^4}\right) \frac{\psi}{\rho}+\left(2\gamma \frac{z_iz_j}{|z|^2}+\delta_{ij}\right) \, \frac{\psi}{\rho}z_r\left[\frac{2\gamma}{|z|^2}-\frac{\psi}{\rho^2}(2\gamma +1)\right]. 
	\end{align*}	
	Since $|z|\leq \rho$ and  $\frac{|z|}{\rho}=\psi^{\frac{1}{2\gamma}}$, we have $|X_r(X_iX_j \rho)|\leq C \left[\frac{\psi^2}{\rho^2}+ \frac{\psi}{\rho |z|} \right] \leq C \frac{\psi^{1-\frac{1}{2\gamma}}}{\rho^2}.$ Thus, we have
	\begin{align} \label{Third1}
	\left|\left(\frac{\rho}{\mu}\sum a_{qr} X_q\rho\right) \, b_{ij}  X_r(X_iX_j \rho)\right| \leq C \, \rho^2 \mu^{\frac{1}{2\gamma}} \, \frac{\psi^{1-\frac{1}{2\gamma}}}{\rho^2} \leq C \psi.
	\end{align}
\vspace{.1in}

	\item  For $m+1\leq r \leq N$ and $1\leq i,j \leq m,$ we have:
	\begin{eqnarray*}	
		&&X_r(X_iX_j \rho)=(2\gamma+1)(\gamma+1) z_iz_j\frac{\psi^2}{\rho^{2\gamma+5}}t_{r-m}[4\gamma |z|^{\gamma}+3 \rho^{\gamma}\psi^{1/2}]\\
		&&-\left((2\gamma\frac{z_iz_j}{|z|^2}+\delta_{ij})(\gamma+1)\right)\left[\frac{\psi}{\rho^{2\gamma+3}}t_{r-m}[2\gamma |z|^{\gamma}+\rho^{\gamma}\psi^{1/2}]\right].
	\end{eqnarray*}	
	Since $|z|\leq \rho$, $\frac{|z|}{\rho}=\psi^{\frac{1}{2\gamma}}$ and $|t|\leq \rho^{\gamma+1}$, we have $|X_r(X_iX_j \rho)|\leq C  \frac{\psi^{1+\frac{1}{2}}}{\rho^2}$. Thus, we have
	\begin{align} \label{Third2}
	\left|\left(\frac{\rho}{\mu}\sum a_{qr} X_q\rho\right) \, b_{ij}  X_r(X_iX_j \rho)\right| \leq  C \psi.
	\end{align}
\vspace{.1in}

	\item  For $1\leq r \leq m$, ~$1\leq i\leq m$ and $1\leq j \leq k,$ we have:
	\begin{align*}
		&X_r(X_iX_{m+j} \rho)=-(2\gamma+1)(\gamma+1)X_r\left(\frac{z_it_j}{|z|^{\gamma}}\right) \frac{\psi^2}{\rho^3}-(2\gamma+1)(\gamma+1)\left(\frac{z_it_j}{|z|^{\gamma}}\right)\frac{\psi^2}{\rho^3}z_r\left[\frac{4\gamma}{|z|^2}-\frac{(4\gamma+3)\psi}{\rho^2}\right]\\
		&+\frac{\psi}{\rho}\left[\gamma(\gamma+1)X_r\left(\frac{z_it_j}{|z|^{\gamma+2}}\right)\right]+\left[\gamma(\gamma+1)\left(\frac{z_it_j}{|z|^{\gamma+2}}\right)\right]\frac{\psi}{\rho}z_r\left[\frac{2\gamma}{|z|^2}-\frac{\psi}{\rho^2}(2\gamma+1)\right]\\
		&=-(2\gamma+1)(\gamma+1)\left(\frac{\delta_{ri}t_j}{|z|^{\gamma}}-\gamma \frac{z_i z_rt_j}{|z|^{\gamma+2}}  \right) \frac{\psi^2}{\rho^3}-(2\gamma+1)(\gamma+1)\left(\frac{z_it_j}{|z|^{\gamma}}\right)\frac{\psi^2}{\rho^3}z_r\left[\frac{4\gamma}{|z|^2}-\frac{(4\gamma+3)\psi}{\rho^2}\right]\\
		&+\frac{\psi}{\rho}\left[\gamma(\gamma+1) \left(\frac{\delta_{ir}t_j}{|z|^{\gamma+2}}-(\gamma+2) \frac{z_i z_rt_j}{|z|^{\gamma+4}}\right)\right]+\left[\gamma(\gamma+1)\left(\frac{z_it_j} {|z|^{\gamma+2}}\right)\right]\frac{\psi}{\rho}z_r\left[\frac{2\gamma}{|z|^2}-\frac{\psi}{\rho^2}(2\gamma+1)\right].
	\end{align*}
	Since $|z|\leq \rho$, $\frac{|z|}{\rho}=\psi^{\frac{1}{2\gamma}}$ and $|t|\leq \rho^{\gamma+1}$, we have $|X_r(X_iX_{m+j} \rho)|\leq C\left[ \frac{\psi^{3/2}}{\rho^{2}}+\frac{\psi^{1/2}}{|z|^2} \right]  \leq C \frac{\psi^{1/2-\frac{1}{\gamma}}}{\rho^2}.$ Thus, we have
	\begin{align} \label{Third3}
	\left|\left(\frac{\rho}{\mu}\sum a_{qr} X_q\rho\right) \, b_{i(m+j)}  X_r(X_iX_{m+j} \rho)\right| \leq C \, \rho^2 \mu^{\frac{1}{2\gamma}} \, (\psi^{1/2+\frac{1}{2\gamma}}) \,\frac{\psi^{1/2-\frac{1}{\gamma}}}{\rho^2} \leq C \psi.
	\end{align}
\vspace{.2in}

%	\item  For $m+1\leq r \leq N$ and $1\leq i,j \leq m,$ we have:
%	\begin{eqnarray*}	
%		&&X_r(X_iX_j \rho)=(2\gamma+1)(\gamma+1) z_iz_j\frac{\psi^2}{\rho^{2\gamma+5}}t_{r-m}[4\gamma |z|^{\gamma}+3 \rho^{\gamma}\psi^{1/2}]\\
%		&&-\left(2\gamma\frac{z_iz_j}{|z|^2}+\delta_{ij})(\gamma+1\right)\left[\frac{\psi}{\rho^{2\gamma+3}}t_{r-m}[2\gamma |z|^{\gamma}+\rho^{\gamma}\psi^{1/2}]\right].
%	\end{eqnarray*}	
%	Since $|z|\leq \rho$, $\frac{|z|}{\rho}=\psi^{\frac{1}{2\gamma}}$ and $|t|\leq \rho^{\gamma+1}$, we have $|X_r(X_iX_j \rho)|\leq C \, \frac{\psi}{\rho |z|} \leq C  \frac{\psi^{1-\frac{1}{2\gamma}}}{\rho^2 }$. Thus, we have
%	$|(\frac{\rho}{\mu}\sum b_{qr} X_q\rho) \, b_{ij}  X_r(X_iX_j \rho)|\leq C \, \rho^2 \mu^{\frac{1}{2\gamma}} \,\frac{\psi^{1-\frac{1}{2\gamma}}}{\rho^2 } \leq C \psi.$
%\vspace{.2in}

	\item For $m+1\leq r \leq N$, ~$1\leq i\leq m$ and $1\leq j \leq k,$ we have:
	\begin{align*}
		&X_r(X_iX_{m+j} \rho)=-(2\gamma+1)(\gamma+1)X_r\left(\frac{z_it_j}{|z|^{\gamma}}\right) \frac{\psi^2}{\rho^3}+(2\gamma+1)(\gamma+1)^2\left(\frac{z_it_j}{|z|^{\gamma}}\right)\frac{\psi^2}{\rho^{2\gamma+5}}t_{r-m}\left[4\gamma |z|^{\gamma}+3 \rho^{\gamma}\psi^{1/2}\right]\\
		&+\frac{\psi}{\rho}\left[\gamma(\gamma+1)X_r\left(\frac{z_it_j}{|z|^{\gamma+2}}\right)\right]-\left[\gamma(\gamma+1)^2\left(\frac{z_it_j}{|z|^{\gamma+2}}\right)\right]\frac{\psi}{\rho^{2\gamma+3}}t_{r-m}\left[2\gamma |z|^{\gamma}+\rho^{\gamma}\psi^{1/2}\right]\\
		&=-(2\gamma+1)(\gamma+1)\left(\frac{z_i |z|^{\gamma} \delta_{rj}}{|z|^{\gamma}}\right) \frac{\psi^2}{\rho^3}+(2\gamma+1)(\gamma+1)^2\left(\frac{z_it_j}{|z|^{\gamma}}\right)\frac{\psi^2}{\rho^{2\gamma+5}}t_{r-m}\left[4\gamma |z|^{\gamma}+3 \rho^{\gamma}\psi^{1/2}\right]\\
		&+\frac{\psi}{\rho}\left[\gamma(\gamma+1)\left(\frac{z_i |z|^{\gamma} \, \delta_{rj}}{|z|^{\gamma+2}}\right)\right]-\left[\gamma(\gamma+1)^2\left(\frac{z_it_j}{|z|^{\gamma+2}}\right)\right]\frac{\psi}{\rho^{2\gamma+3}}t_{r-m}\left[2\gamma |z|^{\gamma}+\rho^{\gamma}\psi^{1/2}\right].
	\end{align*}
Since $|z|\leq \rho$, $\frac{|z|}{\rho}=\psi^{\frac{1}{2\gamma}}$ and $|t|\leq \rho^{\gamma+1}$, we have $|X_r(X_iX_{m+j} \rho)|\leq C \, \left[ \frac{\psi^2}{\rho^2}+ \frac{\psi}{\rho |z|} \right] \leq C \frac{\psi^{1-\frac{1}{2\gamma}}}{\rho^2 }$. Thus, we have
\begin{align} \label{Third4}
	\left|\left(\frac{\rho}{\mu}\sum a_{qr} X_q\rho\right) \, b_{i (m+j)}  X_r(X_iX_{m+j} \rho)\right|\leq C \, \rho^2 \mu^{-\frac{1}{2}} \, (\psi^{1/2+\frac{1}{2\gamma}}) \, \frac{\psi^{1-\frac{1}{2\gamma}}}{\rho^2 } \leq C \psi.
	\end{align}
\vspace{.2in}

	\item For $1\leq r,i\leq m$ and $1\leq j\leq k$ we have:
	\begin{align*}
		&X_r(X_{m+j} X_i\rho)=-(2\gamma+1)(\gamma+1) X_r\left(\frac{z_it_j}{|z|^{\gamma}}\right) \, \frac{\psi^2}{\rho^3}-(2\gamma+1)(\gamma+1) \left(\frac{z_it_j}{|z|^{\gamma}}\right) \, \frac{\psi^2}{\rho^3}z_r\left[\frac{4\gamma}{|z|^2}-\frac{(4\gamma+3)\psi}{\rho^2}\right]\\
		&=-(2\gamma+1)(\gamma+1) \left(\frac{\delta_{ri}t_j}{|z|^{\gamma}}-\gamma \frac{z_iz_rt_j}{|z|^{\gamma+2}}\right) \, \frac{\psi^2}{\rho^3}-(2\gamma+1)(\gamma+1) \left(\frac{z_it_j}{|z|^{\gamma}}\right) \, \frac{\psi^2}{\rho^3}z_r\left[\frac{4\gamma}{|z|^2}-\frac{(4\gamma+3)\psi}{\rho^2}\right].
	\end{align*}
	Since $|z|\leq \rho$, $\frac{|z|}{\rho}=\psi^{\frac{1}{2\gamma}}$ and $|t|\leq \rho^{\gamma+1}$, we have $|X_r(X_{m+j}X_i \rho)|\leq  C \frac{\psi^{3/2}}{\rho^2}$. Thus, we have
	\begin{align} \label{Third5}
	\left|\left(\frac{\rho}{\mu}\sum a_{qr} X_q\rho\right) \, b_{(m+j)i}  X_r(X_{m+j}X_i \rho)\right|\leq C \, \rho^2  \,\frac{\psi^{3/2}}{\rho^2}\leq C \psi.
	\end{align}
\vspace{.2in}

	\item For $m+1\leq r\leq N$, $1\leq i\leq m$ and $1\leq j\leq k$ we have:
	\begin{align*}
		&X_r(X_{m+j} X_i\rho)=-(2\gamma+1)(\gamma+1) X_r\left(\frac{z_it_j}{|z|^{\gamma}}\right) \, \frac{\psi^2}{\rho^3}+(2\gamma+1)(\gamma+1)^2 \left(\frac{z_it_j}{|z|^{\gamma}}\right) \, \frac{\psi^2}{\rho^{2\gamma+5}}t_{r-m}\left[4\gamma |z|^{\gamma}+3 \rho^{\gamma}\psi^{1/2}\right]\\
		&=-(2\gamma+1)(\gamma+1) \left(\frac{z_i \, |z|^{\gamma} \delta_{rj}}{|z|^{\gamma}}\right) \, \frac{\psi^2}{\rho^3}+(2\gamma+1)(\gamma+1)^2 \left(\frac{z_it_j}{|z|^{\gamma}}\right) \, \frac{\psi^2}{\rho^{2\gamma+5}}t_{r-m}\left[4\gamma |z|^{\gamma}+3 \rho^{\gamma}\psi^{1/2}\right].
	\end{align*}
	Since $|z|\leq \rho$, $\frac{|z|}{\rho}=\psi^{\frac{1}{2\gamma}}$ and $|t|\leq \rho^{\gamma+1}$, we have $|X_r(X_{m+j}X_i \rho)|\leq C  \frac{\psi^{3/2}}{\rho^2}$. Thus, we have
	\begin{align} \label{Third6}
	\left|\left(\frac{\rho}{\mu}\sum a_{qr} X_q\rho\right) \, b_{(m+j) i}  X_r(X_{m+j}X_i \rho)\right|\leq C \, \rho^2  \psi^{-1/2} \frac{\psi^{3/2}}{\rho^2} \leq C \psi.
	\end{align}
\vspace{.2in}

	\item For $1\leq r\leq m$ and $1\leq i, j\leq k$ we have:
	\begin{align*}
		&X_r(X_{m+i}X_{m+j} \rho)=-(2\gamma+1)(\gamma+1)^2 X_r\left(\frac{t_jt_i}{|z|^{2\gamma}}\right) \, \frac{\psi^2}{\rho^3}-(2\gamma+1)(\gamma+1)^2 \left(\frac{t_jt_i}{|z|^{2\gamma}}\right) \, \frac{\psi^2}{\rho^3}z_r\\
		&\left[\frac{4\gamma}{|z|^2}-\frac{(4\gamma+3)\psi}{\rho^2}\right]	+(\gamma+1) \delta_{ij} \frac{\psi}{\rho} \, z_r\left[\frac{2\gamma}{|z|^2}-\frac{\psi}{\rho^2}(2\gamma+1)\right]\\
		&=-(2\gamma+1)(\gamma+1)^2 \left(-2\gamma \frac{t_iz_rt_i}{|z|^{2\gamma+1}}\right) \, \frac{\psi^2}{\rho^3}-(2\gamma+1)(\gamma+1)^2 \left(\frac{t_jt_i}{|z|^{2\gamma}}\right) \, \frac{\psi^2}{\rho^3}z_r\\
		&\left[\frac{4\gamma}{|z|^2}-\frac{(4\gamma+3)\psi}{\rho^2}\right]	+(\gamma+1) \delta_{ij} \frac{\psi}{\rho} \, z_r\left[\frac{2\gamma}{|z|^2}-\frac{\psi}{\rho^2}(2\gamma+1)\right].
	\end{align*}
	Since $|z|\leq \rho$, $\frac{|z|}{\rho}=\psi^{\frac{1}{2\gamma}}$ and $|t|\leq \rho^{\gamma+1}$, we have $|X_r(X_{m+i}X_{m+j} \rho)|\leq C \, \left[\frac{\psi}{\rho} +\frac{\psi}{\rho |z|}\right]\leq C \frac{\psi^{1-\frac{1}{2\gamma}}}{\rho^2}$. Thus, we have
	\begin{align} \label{Third7}
	\left|\left(\frac{\rho}{\mu}\sum a_{qr} X_q\rho\right) \, b_{(m+i)(m+j)}  X_r(X_{m+i}X_{m+j} \rho)\right|\leq C \, \rho^2 \mu^{\frac{1}{2\gamma}} \, \frac{\psi^{1-\frac{1}{2\gamma}}}{\rho^2} \leq C \psi.
	\end{align}
\vspace{.2in}

	\item For $m+1\leq r\leq N$ and $1\leq i, j\leq k$ we have:
	\begin{align*}
		&	X_r(X_{m+i} X_{m+j}\rho)=-(2\gamma+1)(\gamma+1)^2 X_r\left(\frac{t_jt_i}{|z|^{2\gamma}}\right) \, \frac{\psi^2}{\rho^3}+(2\gamma+1)(\gamma+1)^3 \left(\frac{t_jt_i}{|z|^{2\gamma}}\right) \, \frac{\psi^2}{\rho^{2\gamma+5}}\\
		&t_{r-m} \left[4\gamma |z|^{\gamma}+3 \, \rho^{\gamma} \, \psi^{1/2}\right]-(\gamma+1)^2 \delta_{ij} \frac{\psi}{\rho^{2\gamma+3}}t_{r-m} \left[2\gamma |z|^{\gamma}+ \rho^{\gamma}\psi^{1/2}\right]\\
		&=-(2\gamma+1)(\gamma+1)^2 \left(\frac{t_j|z|^{\gamma} \delta_{ri}}{|z|^{2\gamma}} + \frac{t_i|z|^{\gamma} \delta_{rj}}{|z|^{2\gamma}} \right) \, \frac{\psi^2}{\rho^3}+(2\gamma+1)(\gamma+1)^3 \left(\frac{t_jt_i}{|z|^{2\gamma}}\right) \, \frac{\psi^2}{\rho^{2\gamma+5}}\\
		&t_{r-m} \left[4\gamma |z|^{\gamma}+3 \, \rho^{\gamma} \, \psi^{1/2}\right]-(\gamma+1)^2 \delta_{ij} \frac{\psi}{\rho^{2\gamma+3}}t_{r-m} \left[2\gamma |z|^{\gamma}+ \rho^{\gamma}\psi^{1/2}\right].
	\end{align*}
\end{enumerate}
Since $|z|\leq \rho$, $\frac{|z|}{\rho}=\psi^{\frac{1}{2\gamma}}$ and $|t|\leq \rho^{\gamma+1}$, we have $|X_r(X_{m+i}X_{m+j} \rho)|\leq C \frac{\psi^{3/2}}{\rho^2}$. Thus, we have
\begin{align} \label{Third8}
	\left|\left(\frac{\rho}{\mu}\sum a_{qr} X_q\rho\right) \, b_{(m+i)(m+j)}  X_r(X_iX_j \rho)\right|\leq C \, \rho^2 \mu^{-\frac{1}{2}} \, \frac{\psi^{3/2}}{\rho^2}  \leq C \psi.
	\end{align}
	The estimate \eqref{third} now follows from \eqref{cal1}-\eqref{Third8}. 
 
\end{proof}

\end{document}